\documentclass[11pt,a4paper,oneside,reqno]{amsart} 
\usepackage[T1]{fontenc} 
\usepackage[english]{babel}
\usepackage[p,osf]{cochineal}
\usepackage[varqu,varl,var0]{inconsolata}
\usepackage[scale=.95,type1]{cabin} 
\usepackage{mathalfa}
\usepackage[utf8]{inputenc} 
\usepackage{amsthm}
\usepackage{amsmath}
\usepackage{amssymb}
\usepackage{amsfonts}
\usepackage{amsaddr}
\usepackage{xspace}
\usepackage{enumitem} 
\usepackage{csquotes}
\usepackage{xcolor}
\usepackage[colorlinks]{hyperref}
\definecolor{intcolor}{HTML}{CA0020}
\definecolor{extcolor}{HTML}{0571B0}
\hypersetup{breaklinks,linkcolor=intcolor,citecolor=intcolor,filecolor=black,urlcolor=extcolor,menucolor=intcolor,runcolor=extcolor}
\usepackage{graphicx}
\usepackage{mathtools}
\mathtoolsset{centercolon}
\usepackage{bm}
\usepackage{tikz}
\usepackage{pgf}
\usetikzlibrary{positioning}
\usepackage{xparse}
\ExplSyntaxOn
\NewDocumentCommand{\interval}{sO{}mm}
 {
  \IfBooleanTF{#1}
   {
    \yannick_interval:NNnnn \left \right { } { #3 } { #4 }
   }
   {
    \yannick_interval:NNnnn \mathopen \mathclose { #2 } { #3 } { #4 }
   }
 }

\cs_new_protected:Nn \yannick_interval:NNnnn
 {
  \str_case:nn { #4 }
   {
    {oo}{#1#3\c_yannick_left_open_tl #5 #2#3\c_yannick_right_open_tl}
    {co}{#1#3[                       #5 #2#3\c_yannick_right_open_tl}
    {oc}{#1#3\c_yannick_left_open_tl #5 #2#3]}
    {cc}{#1#3[                       #5 #2#3]}
   }
 }
\tl_const:Nn \c_yannick_left_open_tl { ( }
\tl_const:Nn \c_yannick_right_open_tl { ) }
\ExplSyntaxOff

\DeclareMathOperator{\Tr}{Tr}

\DeclareMathOperator{\E}{\mathbf{E}}
\DeclareMathOperator{\Prob}{\mathbf{P}}

\DeclareMathOperator{\Spec}{Spec}
\DeclareMathOperator{\erf}{erf}

\newcommand{\ii}{\mathrm{i}}

\ifdefined\C
\renewcommand{\C}{\mathbf{C}}
\else
\newcommand{\C}{\mathbf{C}}
\fi
\newcommand{\HC}{\mathbf{H}}

\newcommand{\vx}{\bm{x}}
\newcommand{\vy}{\bm{y}}
\newcommand{\Gin}{\mathrm{Gin}}

\newcommand{\wt}{\widetilde}
\newcommand{\wh}{\widehat}

\newcommand{\R}{\mathbf{R}}
\newcommand{\F}{\mathbf{F}}
\newcommand{\1}{\bm{1}}
\newcommand{\N}{\mathbf{N}}

\newcommand{\DD}{\mathbf{D}}

\newcommand{\cO}{\mathcal{O}}
\newcommand{\co}{{\scriptstyle\mathcal{O}}}

\newcommand{\diff}{\operatorname{d}\!{}}
\DeclarePairedDelimiter{\braket}{\langle}{\rangle}%
\DeclarePairedDelimiter{\abs}{\lvert}{\rvert}%
\DeclarePairedDelimiter{\norm}{\lVert}{\rVert}%
\providecommand\given{}
\newcommand\SetSymbol[1][]{\nonscript\:#1\vert\allowbreak\nonscript\:\mathopen{}}
\DeclarePairedDelimiterX{\tuple}[1](){\renewcommand\given{\SetSymbol[\delimsize]}#1}
\DeclarePairedDelimiterX{\set}[1]{\{}{\}}{\renewcommand\given{\SetSymbol[\delimsize]}#1}
\DeclarePairedDelimiterX{\Set}[1]\{\}{\renewcommand\given{\SetSymbol[\delimsize]}#1}
\DeclarePairedDelimiterXPP{\landauO}[1]{\cO}(){}{#1}
\DeclarePairedDelimiterXPP{\landauo}[1]{\co}(){}{#1}
\DeclarePairedDelimiterXPP{\landauok}[1]{\co_k}(){}{#1}
\DeclarePairedDelimiterXPP{\landauOprec}[1]{\cO_\prec}(){}{#1}
\DeclarePairedDelimiterXPP{\Exp}[1]{\E}[]{}{\renewcommand\given{\SetSymbol[\delimsize]}#1}
\DeclareFontFamily{U}{mathx}{\hyphenchar\font45}
\DeclareFontShape{U}{mathx}{m}{n}{
      <5> <6> <7> <8> <9> <10>
      <10.95> <12> <14.4> <17.28> <20.74> <24.88>
      mathx10
      }{}
\DeclareSymbolFont{mathx}{U}{mathx}{m}{n}
\DeclareFontSubstitution{U}{mathx}{m}{n}
\DeclareMathAccent{\widecheck}{0}{mathx}{"71}
\usepackage[giveninits=true,url=false,doi=false,isbn=false,eprint=true,datamodel=mrnumber,sorting=nty,maxcitenames=4,maxbibnames=99,backref=false,block=space,backend=biber,bibstyle=phys]{biblatex} 
\AtEveryBibitem{\clearfield{month}}
\AtEveryCitekey{\clearfield{month}} 
\renewbibmacro{in:}{}
\DeclareFieldFormat[article]{title}{\emph{#1}} 
\DeclareFieldFormat{mrnumber}{\ifhyperref{\href{http://www.ams.org/mathscinet-getitem?mr=#1}{\nolinkurl{MR#1}}}{\nolinkurl{#1}}}
\DeclareFieldFormat{pmid}{\ifhyperref{\href{https://www.ncbi.nlm.nih.gov/pubmed/#1}{\nolinkurl{PMID#1}}}{\nolinkurl{#1}}}
\DeclareFieldFormat{eprint}{\ifhyperref{\href{https://arxiv.org/abs/#1}{\nolinkurl{arXiv:#1}}}{\nolinkurl{#1}}}
\renewbibmacro*{doi+eprint+url}{%
  \iftoggle{bbx:doi}{\printfield{doi}}{}
  \newunit\newblock%
  \printfield{mrnumber}%
  \newunit\newblock%
  \printfield{pmid}%
  \newunit\newblock%
  \printfield{eprint}%
  \iftoggle{bbx:url}{\usebibmacro{url+urldate}}{}}
\bibliography{references}
\date{\today}
\author{Giorgio Cipolloni\(^{\dagger\ddagger}\) \and L\'aszl\'o Erd\H{o}s\(^{\dagger}\)}
\address{IST Austria, Am Campus 1, A-3400 Klosterneuburg, Austria}
\author{Dominik Schr\"oder\(^{\dagger}\)}
\address{Institute for Theoretical Studies, ETH Zurich, Clausiusstr.\ 47, 8092 Zurich, Switzerland}
\email{dschroeder@ethz.ch}
\email{gcipollo@ist.ac.at} 
\email{lerdos@ist.ac.at} 
\thanks{\(^\dagger\)Partially supported by ERC Advanced Grant No.~338804}\thanks{\(^\ddagger\)This project has received funding from the European Union's Horizon 2020 research and innovation programme under the Marie Sk\l odowska-Curie Grant Agreement No. 665385.}  
\subjclass[2010]{60B20, 15B52} 
\keywords{Ginibre ensemble, Circular law, Girko's formula}
\title{Edge Universality for non-Hermitian Random Matrices}
\date{\today}

\newtheorem{theorem}{Theorem}
\newtheorem{assumption}{Assumption}

\newtheorem{lemma}{Lemma}

\newtheorem{proposition}{Proposition}
\newtheorem{remark}{Remark}

\begin{document}
\thispagestyle{empty}

\begin{abstract}
  We consider large non-Hermitian real or complex random matrices  \(X\) with 
  independent,  identically distributed centred entries. We
  prove that their local eigenvalue statistics near the spectral edge, 
  the unit circle, coincide with   those of the Ginibre ensemble,
  i.e.\ when the matrix elements of \(X\) are Gaussian. This result is the 
  non-Hermitian counterpart of the universality of the Tracy-Widom 
  distribution at the spectral edges of the Wigner ensemble. 
\end{abstract}

\maketitle

\section{Introduction}
Following Wigner's  motivation from physics,
most universality results  on the local eigenvalue statistics for large random matrices 
concern the Hermitian case. In particular, the celebrated Wigner-Dyson statistics 
in the bulk spectrum~\cite{MR0220494}, the Tracy-Widom statistics~\cite{MR1257246,MR1385083} at the spectral edge
and the Pearcey statistics~\cite{MR0020857, MR2207649} at
the possible cusps of the eigenvalue density profile all describe eigenvalue statistics
of a large Hermitian random matrix. In the last decade there has been a spectacular progress 
in verifying Wigner's original vision, formalized as the Wigner-Dyson-Mehta conjecture, for Hermitian ensembles with increasing generality, 
see e.g.~\cite{MR1810949, MR2375744, MR2662426, MR2784665, MR2810797, MR3687212, MR3800840,
MR3502606, MR2784665, MR3719056, MR3729630, MR3941370,
MR2726110, 1807.01559}
for the bulk,~\cite{MR2012268, MR1727234, MR2669449, MR3695802, MR3405746, MR3429490, MR3253704, 1712.03881, MR4089499} for the edge 
and more recently~\cite{MR4134946,MR3485343, MR4026551} at the cusps. 

Much less is known about the spectral universality for non-Hermitian models. In the simplest case of
the Ginibre ensemble, i.e.\ random matrices with i.i.d.\ standard Gaussian entries without any symmetry condition,
explicit formulas for all correlation functions 
have been computed first for the complex case~\cite{MR173726}
and later  for the more complicated  real case~\cite{MR2530159, MR2439268,  MR2185860} (with special cases  solved earlier~\cite{MR1121461, MR1437734, MR1231689}).
Beyond the explicitly computable Ginibre case only the method of \emph{four moment matching} by Tao and Vu
has been available. 
Their main universality result in~\cite{MR3306005} states that the local correlation functions
of the eigenvalues of a random matrix \(X\)  with i.i.d.\ matrix elements
coincide with those of the Ginibre ensemble as long as the first four moments
of the common distribution of the entries of \(X\) (almost) match 
the first four moments of the standard Gaussian. This result holds for
both real and complex cases as well as throughout the spectrum, including the edge regime.

In the current paper we prove the edge  universality for any \(n\times n\) random matrix \(X\) with centred  i.i.d.\ entries
in the edge regime, in particular we remove the four moment matching condition from~\cite{MR3306005}.
More precisely, under the normalization \(\E \abs{x_{ab}}^2= \frac{1}{n}\), the spectrum of \(X\) converges
to the unit disc with a uniform spectral density according to 
the \emph{circular law}~\cite{MR773436,MR1428519,MR2409368,MR866352,MR863545,MR3813992}.
The typical distance between nearest eigenvalues is of order \(n^{-1/2}\). We pick a reference point \(z\) on 
the boundary of the limiting spectrum, \(\abs{z}=1\),
and rescale  correlation functions by a factor of \(n^{-1/2}\)
to detect the correlation of  individual eigenvalues. We show that these rescaled correlation functions
converge to those of the Ginibre ensemble as \(n\to \infty\). This result is the non-Hermitian
analogue of the Tracy-Widom edge universality in the Hermitian case. 
A similar result is expected to hold in the bulk regime, i.e.\ for any reference point   \(\abs{z}<1\), but 
our method is currently restricted to the edge.

Investigating spectral statistics of non-Hermitian random matrices is considerably more challenging
than Hermitian ones. We give two fundamental reasons for this: the first one is already present
in the proof of the circular law on the global scale. The second one is specific to the 
most powerful existing method to  prove universality of eigenvalue fluctuations.

The first issue a general one; it is well known that
non-Hermitian, especially non-normal spectral analysis is difficult  because, unlike in the Hermitian case, the 
resolvent~\((X-z)^{-1}\)  of a non-normal matrix
is not effective to study eigenvalues near \(z\). Indeed, \((X-z)^{-1}\) can be very large
even if \(z\) is away from the spectrum, a fact that is closely related to the instability of the non-Hermitian eigenvalues under perturbations.
The only useful expression to 
grasp non-Hermitian eigenvalues is Girko's celebrated formula, see~\eqref{eq girko} later,
expressing linear statistics of eigenvalues of \(X\) in terms of the log-determinant
of the  symmetrized  matrix 
\begin{equation}
  \label{eq:linmatrix}
  H^z= \left( \begin{matrix}
    0 & X-z \\
    X^*-\overline{z} & 0
  \end{matrix} \right).
\end{equation}
Girko's formula is much more subtle and harder to analyse than the analogous 
expression for the Hermitian case involving the boundary value of the
resolvent on the real line. In particular, it requires a good lower bound on the
smallest singular value of \(X-z\), a notorious  difficulty behind 
the proof of the  circular law.  Furthermore, any conceivable universality proof would  rely on a
local version of the circular law as an a priori control. Local  laws on optimal scale assert 
that the  eigenvalue density on a scale \(n^{-1/2+\epsilon}\)
is deterministic with high probability, i.e.\ it is a law of large number type result and is not sufficiently
refined to detect correlations of  individual eigenvalues. The 
proof of the local circular law requires a careful analysis of \(H^z\) that has an additional structural 
instability due to its block symmetry. A specific estimate, tailored to Girko's formula, on the trace of the 
resolvent of \((H^z)^2\)  was the main ingredient behind the proof of the local circular law on optimal scale~\cite{MR3230002, MR3230004, MR3278919}, see also~\cite{MR3306005} under three moment matching condition.
Very recently the optimal local circular law was even proven
for ensembles with  inhomogeneous variance profiles in the bulk~\cite{MR3770875} and at the
edge~\cite{1907.13631}, the latter result also gives an optimal control on the spectral radius. An optimal local law for \(H^z\) in the edge regime previously had not been available, even in the i.i.d.\ case.

The second major obstacle to prove universality of  fluctuations of non-Hermitian eigenvalues is the lack of a good 
analogue of the Dyson Brownian motion. The essential ingredient behind the strongest universality results 
in the Hermitian case  is  the Dyson  Brownian motion (DBM)~\cite{MR148397},
a system of coupled stochastic differential equations  (SDE) that the eigenvalues of a natural stochastic
flow of random matrices satisfy, see~\cite{MR3699468} for a pedagogical summary. 
The corresponding SDE  in the non-Hermitian case involves 
not only eigenvalues but overlaps of eigenvectors as well, see e.g.~\cite[Appendix A]{MR4095019}.
Since overlaps themselves have strong correlation whose proofs are highly nontrivial 
even in the Ginibre case~\cite{MR3851824, MR4095019}, the 
analysis of  this SDE is currently beyond reach.

Our proof of the edge universality circumvents DBM and it has two key ingredients. 
The first  main input is an optimal local law
for the resolvent of \(H^z\) both in \emph{isotropic} and \emph{averaged} sense, see~\eqref{eq local law} later, that
allows for a concise and transparent  comparison of the joint distribution of several resolvents
of \(H^z\) with their Gaussian counterparts by following their evolution 
under the natural Ornstein-Uhlenbeck (OU). We are able to control this  flow for a long time,  similarly
to an earlier proof of the Tracy-Widom law at the spectral edge of a Hermitian ensemble~\cite{MR3582818}. Note that the density of eigenvalues
of  \(H^z\) develops a cusp as \(\abs{z}\) passes through 1, the spectral radius of \(X\).  
The optimal local law for very general Hermitian ensembles in the cusp regime
has recently been proven~\cite{MR4134946}, strengthening the non-optimal result
in~\cite{MR3719056}. This optimality was essential in the proof of the universality of the
Pearcey statistics for both the complex Hermitian~\cite{MR4134946} and real symmetric~\cite{MR4026551} 
matrices with a cusp in their density of states. The matrix \(H^z\), however,
does not satisfy the  key \emph{flatness} condition required~\cite{MR4134946} due its large zero blocks.
A very delicate analysis of the underlying matrix Dyson equation was necessary
to overcome the flatness condition and prove the optimal local law for \(H^z\) in~\cite{MR3770875,1907.13631}.

Our  second key input is a lower tail estimate on the lowest singular value of \(X-z\) when \(\abs{z}\approx 1\).
A very mild regularity assumption on the distribution of the matrix elements of \(X\), see~\eqref{eq:smbound} later, guarantees that there
is no singular value below \(n^{-100}\), say. Cruder bounds  guarantee that there cannot be  more than \(n^\epsilon\) singular values 
below \(n^{-3/4}\); note that this natural scaling reflects the cusp at zero in the density of states of \(H^z\).
Such information on the possible singular values in the regime \([n^{-100}, n^{-3/4}]\) is sufficient
for the optimal local law since it  is insensitive to \(n^\epsilon\)-eigenvalues, but for universality every eigenvalue must be
accounted for. We therefore need a stronger lower tail bound on the lowest eigenvalue \(\lambda_1\) of 
\((X-z)(X-z)^*\). With supersymmetric methods  we recently proved~\cite{1908.01653} a precise bound   of the form
\begin{equation}\label{our lambda}
  \Prob\Big(\lambda_1 \big( (X-z)(X-z)^*\big)  \le \frac{x}{n^{3 /2}}\Big)\lesssim \begin{cases}
    x+ \sqrt{x} e^{- n(\Im z)^2},& X\sim \Gin(\R)\\
    x,& X\sim\Gin(\C),
  \end{cases} 
\end{equation}  
modulo logarithmic corrections, for the Ginibre ensemble  whenever \(\abs{z}=1 + \landauO{n^{-1/2}}\). 
Most importantly,~\eqref{our lambda} controls \(\lambda_1\)  on the optimal \(n^{-3/2}\) scale
and thus excluding singular values in the intermediate regime \([n^{-100}, n^{-3/4-\epsilon}]\)
that was inaccessible with other methods. We extend this control to \(X\) with i.i.d.\ entries
from the Ginibre ensemble with Green function comparison argument
using again the  optimal local law for \(H^z\).

\subsection*{Notations and conventions}
We introduce some notations we use throughout the paper. We write \(\HC\) for the upper half-plane \(\HC:=\set{z\in\C\given\Im z>0}\), and for any \(z\in\mathbb{C}\) we use the notation \(\diff z:= 2^{-1} \ii(\diff z\wedge \diff \overline{z})\) for the two dimensional volume form on \(\mathbb{C}\). For any \(2n\times 2n\) matrix \(A\) we use the notation \(\langle A\rangle:= (2n)^{-1}\Tr  A\) to denote the normalized trace of \(A\). For positive quantities \(f,g\) we write \(f\lesssim g\) and \(f\sim g\) if \(f \le C g\) or \(c g\le f\le Cg\), respectively, for some constants \(c,C>0\) which depends only on the constants appearing in~\eqref{eq:boundmom}. We denote vectors by bold-faced lower case Roman letters \({\bm x}, {\bm y}\in\C^k\), for some \(k\in\N\). Vector and matrix norms, \(\norm{\vx}\) and \(\norm{A}\), indicate the usual Euclidean norm and the 
corresponding induced matrix norm. Moreover, for a vector \({\bm x}\in\C^k\), we use the notation \(\diff {\bm x}:= \diff x_1\dots \diff x_k\).

We will use the concept of ``with very high probability'' meaning that for any fixed \(D>0\) the probability of the event is bigger than \(1-n^{-D}\) if \(n\ge n_0(D)\). Moreover, we use the convention that \(\xi>0\) denotes an arbitrary small constant.

We use the convention that quantities without tilde refer to a general matrix with i.i.d.\ entries, whilst any quantity with tilde refers to the Ginibre ensemble, e.g.\ we use \(X\), \(\{\sigma_i\}_{i=1}^n\) to denote a non-Hermitian matrix with i.i.d.\ entries and its eigenvalues, respectively, and \(\widetilde{X}\), \(\{\widetilde{\sigma}_i\}_{i=1}^n\) to denote their Ginibre counterparts.

\section{Model and main results}
We consider real or complex i.i.d.\ matrices \(X\), i.e.\ matrices whose entries are independent and identically distributed as \(x_{ab} \stackrel{d}{=} n^{-1/2}\chi\) for a random variable \(\chi\). We formulate two assumptions on the random variable \(\chi\):
\begin{assumption}\label{as:bm}
  In the real case we assume that \(\E \chi=0\) and \(\E\chi^2=1\), while in the complex case we assume \(\E\chi=\E\chi^2=0\) and \(\E\abs{\chi}^2=1\). In addition, we assume the existence of high moments, i.e.\ that there exist constants \(C_p>0\) for each \(p\in\N\), such that
  \begin{equation}
    \label{eq:boundmom}
    \E \abs{\chi}^p\le C_p.
  \end{equation}
\end{assumption}

\begin{assumption}\label{as:smd}
  There exist \(\alpha,\beta>0\) such that the probability density \(g\colon \F\to[0,\infty)\) of the random variable \(\chi\) satisfies
  \begin{equation}
    \label{eq:smbound}
    g \in L^{1+\alpha}(\F),\qquad \lVert g\rVert_{1+\alpha}\le n^\beta,
  \end{equation}
  where \(\F=\R,\C\) in the real and complex case, respectively. 
\end{assumption}
\begin{remark}\label{remm:tv}
  We remark that we use Assumption~\ref{as:smd} only to 
  control the probability of a very small singular value of \(X-z\).
  Alternatively, one may use 
  the statement
  \begin{equation}
    \label{eq:tv}
    \Prob(  \Spec(H^z) \cap [-n^{-l}, n^{-l}] =\emptyset)\le C_l n^{-l/2},
  \end{equation}
  for any \(l\ge 1\), uniformly in \(\abs{z}\le 2\), 
  that follows directly from~\cite[Theorem 3.2]{MR2684367} without Assumption~\ref{as:smd}. 
  Using~\eqref{eq:tv} makes Assumption~\ref{as:smd} superfluous in the entire paper, albeit at the expense of 
  a quite sophisticated proof.
\end{remark}

We denote the eigenvalues of \(X\) by \(\sigma_1,\dots,\sigma_n\in\C\), and define the \emph{\(k\)-point correlation function \(p_k^{(n)}\)} of \(X\) implicitly such that
\begin{equation}
  \label{eq:kpointfunc}
  \begin{split}
    &\int_{\C  ^k}F(z_1,\dots, z_k) p_k^{(n)}(z_1,\dots, z_k)\, \diff z_1\dots \diff z_k\\
    &\qquad=\binom{n}{k}^{-1}\E  \sum_{i_1,\dots, i_k} F(\sigma_{i_1}, \dots, \sigma_{i_k}),
  \end{split}
\end{equation}
for any smooth compactly supported test function \(F\colon\C^k \to \C\), with \(i_j\in \{1,\dots, n\}\) for \(j\in\{1,\dots,k\}\) all distinct. For the important special case when \(\chi\) follows a standard real or complex Gaussian distribution, we denote the \(k\)-point function of the \emph{Ginibre matrix} \(X\) by \(p_k^{(n,\Gin(\F))}\) for \(\F=\R,\C\). The \emph{circular law} implies that the \(1\)-point function converges 
\[\lim_{n\to\infty }p_1^{(n)}(z) = \frac{1}{\pi} \1(z\in\DD) = \frac{1}{\pi}\1(\abs{z}\le 1) \] 
to the uniform distribution on the unit disk. On the scale \(n^{-1/2}\) of individual eigenvalues the scaling limit of the \(k\)-point function has been explicitly computed in the case of complex and real Ginibre matrices, \(X\sim\Gin(\R),\Gin(\C)\), i.e.\ for any fixed \(z_1,\dots,z_k,w_1,\dots,w_k\in\C\) there exist scaling limits \(p_{z_1,\dots,z_k}^{(\infty)}=p_{z_1,\dots,z_k}^{(\infty,\Gin(\F))}\) for \(\F=\R,\C\) such that
\begin{equation}\label{ginibre kpt limit} 
  \lim_{n\to\infty} p_k^{(n,\Gin(\F))}\Bigl(z_1+\frac{w_1}{n^{1/2}},\dots,z_k+\frac{w_k}{n^{1/2}}\Bigr) = p_{z_1,\dots,z_k}^{(\infty,\Gin(\F))}(w_1,\dots,w_k).
\end{equation}
\begin{remark}
  The \(k\)-point correlation function \(p_{z_1,\dots, z_k}^{(\infty,\Gin(\F))}\) of the Ginibre ensemble in both the complex and real cases \(\F=\C,\R\) is explicitly known; see~\cite{MR173726} and~\cite{MR0220494} for the complex case, and~\cite{MR2530159,MR1437734,17930739} for the real case, where the appearance of \(\sim n^{1/2}\) real eigenvalues causes a singularity in the density. In the complex case \(p_{z_1,\dots, z_k}^{(\infty,\Gin(\C))}\) is determinantal, i.e.\ for any \(w_1,\dots, w_k\in\C\) it holds
  \[
  p_{z_1,\dots, z_k}^{(\infty,\Gin(\C))}(w_1,\dots,w_k)=\det \left(K_{z_i,z_j}^{(\infty,\Gin(\C))}(w_i,w_j)\right)_{1\le i,j\le k}
  \]
  where for any complex numbers \(z_1\), \(z_2\), \(w_1\), \(w_2\) the kernel \(K_{z_1,z_2}^{(\infty,\Gin(\C))}(w_1,w_2)\) is defined by
  \begin{enumerate}[label=(\roman*)]
    \item For \(z_1\ne z_2\), \(K_{z_1,z_2}^{(\infty,\Gin(\C))}(w_1,w_2)=0\).
    \item For \(z_1=z_2\) and \(\abs{z_1}>1\), \(K_{z_1,z_2}^{(\infty,\Gin(\C))}(w_1,w_2)=0\).
    \item For \(z_1=z_2\) and \(\abs{z_1}<1\),
    \[
    K_{z_1,z_2}^{(\infty,\Gin(\C))}(w_1,w_2)=\frac{1}{\pi}e^{-\frac{\abs{w_1}^2}{2}-\frac{\abs{w_2}^2}{2}+w_1\overline{w_2}}.
    \]
    \item For \(z_1=z_2\) and \(\abs{z_1}=1\),
    \[
    K_{z_1,z_2}^{(\infty,\Gin(\C))}(w_1,w_2)=\frac{1}{2\pi}\left[ 1+ \erf\left( -\sqrt{2}(z_1\overline{w_2}+w_1\overline{z_2})\right) \right] e^{-\frac{\abs{w_1}^2}{2}-\frac{\abs{w_2}^2}{2}+w_1\overline{w_2}},
    \]
    where
    \[
    \erf(z):= \frac{2}{\sqrt{\pi}}\int_{\gamma_z} e^{-t^2}\, \diff t,
    \]
    for any \(z\in\C  \), with \(\gamma_z\) any contour from \(0\) to \(z\).
  \end{enumerate}
  For the corresponding much more involved formulas for \(p_k^{(\infty,\Gin(\R))}\) we refer the reader to~\cite{MR2530159}. 
\end{remark}

Our main result is the universality of \(p_{z_1,\dots,z_k}^{(\infty,\Gin(\R,\C))}\) at the edge. In particular we show, that the edge-scaling limit of \(p_k^{(n)}\) agrees with the known scaling limit of the corresponding real or complex Ginibre ensemble.

\begin{theorem}[Edge universality]\label{theo:edgeuniv}
  Let \(X\) be an i.i.d.\ \(n\times n\) matrix, whose entries satisfy Assumption~\ref{as:bm}. 
  Then, for any fixed integer \(k\ge 1\), and complex spectral parameters \(z_1, \dots, z_k\) such that \(\abs{z_j}^2=1\), \(j=1,\dots,k\), and for any compactly supported smooth function \(F\colon\C^k \to \C\), we have the bound
  \begin{equation}
    \label{eq:univ}
    \int_{\C^k} F({\bf w})\left[  p_k^{(n)} \left({\bf z}+\frac{{\bf w}}{\sqrt{n}}\right)-p_{{\bf z}}^{(\infty,\Gin(\F))}({\bf w}) \right]\, \diff {\bf w}=\landauO{ n^{-c}},
  \end{equation}
  where the constant in \(\landauO{\cdot}\) may depend on \(k\) and the \(C^{2k+1}\) norm of \(F\), and \(c>0\) is a small constant depending on \(k\).
\end{theorem}

\subsection{Proof strategy} For the proof of Theorem~\ref{theo:edgeuniv} it is essential to study the linearized \(2n\times 2n\) matrix \(H^z\) defined in~\eqref{eq:linmatrix} with eigenvalues \(\lambda_1^z\le\dots\le\lambda_{2n}^z\) and resolvent \(G(w)=G^z(w):= (H^z-w)^{-1}\). We note that the block structure of \(H^z\) induces a spectrum symmetric around \(0\), i.e.\ \(\lambda_i^z=-\lambda_{2n-i+1}^z\) for \(i=1,\dots,n\). The resolvent becomes approximately deterministic as \(n\to\infty\) and its limit can be found by solving the simple scalar equation 
\begin{equation}
  \label{eq:defhatm}
  -\frac{1}{\widehat{m}^z}= w + \widehat{m}^z- \frac{\abs{z}^2}{w+\widehat{m}^z},\quad \wh m^z(w)\in\HC,\quad w\in\HC,
\end{equation}
which is a special case of the \emph{matrix Dyson equation (MDE)}, see e.g.~\cite{MR3916109}. In the following we may often omit the \(z\)-dependence of \(\widehat{m}^z\), \(G^z(w)\), \(\dots\), in the notation. We note that on the imaginary axis we have \(\wh m(\ii\eta)=\ii\Im \wh m(\ii\eta)\), and in the edge regime \(\abs[\big]{1-\abs{z}^2}\lesssim n^{-1/2}\) we have the scaling~\cite[Lemma~3.3]{1907.13631}
\begin{equation}\label{m scaling} 
  \Im \wh m(\ii\eta) \sim \left.\begin{cases}
    \abs[\big]{1-\abs{z}^2}^{1/2} + \eta^{1/3}, & \abs{z}\le 1,\\
    \frac{\eta}{\abs{1-\abs{z}^2}+\eta^{2/3}}, & \abs{z}>1 
  \end{cases}\right\}  \lesssim n^{-1/4}+\eta^{1/3}.
\end{equation}
For \(\eta>0\) we define 
\begin{equation}\label{M matrix}
  u=u^z(\ii\eta):= \frac{\Im \wh m(\ii\eta)}{\eta+ \Im \wh m(\ii\eta)},\qquad M=M^z(\ii\eta):=  \begin{pmatrix}\wh m(\ii\eta)& -z u(\ii\eta)\\ -\overline{z} u(\ii\eta) & \wh m(\ii\eta) \end{pmatrix}, 
\end{equation}
where \(M\) should be understood as a \(2n\times 2n\) whose four \(n\times n\) blocks are all multiples of the identity matrix, and we note that~\cite[Eq.~(3.62)]{1907.13631}
\begin{equation}\label{M Mp bounds} u(\ii\eta)\lesssim 1,\quad \norm{M(\ii\eta)}\lesssim 1,\quad \norm{M'(\ii\eta)}\lesssim \frac{1}{\eta^{2/3}}  \end{equation} 

Throughout the proof we shall make use of the following optimal local law which is a direct consequence of~\cite[Theorem 5.2]{1907.13631} (extending~\cite[Theorem 5.2]{MR3770875} to the edge regime). Compared to~\cite{1907.13631} we require the local law simultaneously in all the spectral parameters \(z,\eta\) and for \(\eta\) slightly below the fluctuation scale \(n^{-3/4}\). We defer the proofs for both extensions to Appendix~\ref{app:aux}.
\begin{proposition}[Local law for \(H^z\)]\label{prop:local law}
  Let \(X\) be an i.i.d.\ \(n\times n\) matrix, whose entries satisfy Assumption~\ref{as:bm} and~\ref{as:smd}, and let \(H^z\) be as in~\eqref{eq:linmatrix}. Then for any deterministic vectors \(\vx,\vy\) and matrix
  \(R\) and any \(\xi>0\) the following holds true with very high probability: Simultaneously for any \(z\) with for \(\abs{1-\abs{z}}\lesssim n^{-1/2}\) and all \(\eta\) such that \(n^{-1}\le\eta\le n^{100}\) we have the bounds 
  \begin{equation}\label{eq local law}
    \begin{split}   
      \abs{\braket{\vx,(G^z(\ii\eta)-M^z(\ii\eta))\vy}} &\le n^\xi \norm{\vx}\norm{\vy}\Bigl(\frac{1}{n^{1/2}\eta^{1/3}}+\frac{1}{n\eta}\Bigr), \\ 
      \abs{\braket{R(G^z(\ii\eta)-M^z(\ii\eta))}}&\le \frac{n^\xi\norm{R}}{n\eta}.
    \end{split}   
  \end{equation}
\end{proposition}
For the application of Proposition~\ref{prop:local law} towards the proof of Theorem~\ref{theo:edgeuniv} the special case of \(R\) being the identity matrix, and \(\vx,\vy\) being either the standard basis vectors, or the vectors \(\bm1_\pm\) of zeros and ones defined later in~\eqref{bm1pm}.

The linearized matrix \(H^z\) can be related to the eigenvalues \(\sigma_i\) of \(X\) via Girko's Hermitization formula~\cite{MR773436,MR3306005} 
\begin{equation}
  \begin{split}
    \frac{1}{n} \sum_i f_{z_0}(\sigma_i) &= \frac{1}{4\pi n}\int_\C \Delta f_{z_0}(z)\log\abs{\det H_z}\diff z\\&= -\frac{1}{4\pi n}\int_\C \Delta f_{z_0}(z)\int_0^\infty \Im \Tr G^z(\ii\eta)\diff\eta\diff z\label{eq girko}
  \end{split}
\end{equation}
for rescaled test functions \(f_{z_0}(z):= n f(\sqrt{n}(z-z_0))\), where \(f\colon\C\to\C\) is smooth and compactly supported. When using~\eqref{eq girko} the small \(\eta\) regime requires additional bounds on the number of small eigenvalues \(\lambda_i^z\) of \(H^z\), or equivalently small singular values of \(X-z\). For very small \(\eta\), say \(\eta\le n^{-100}\), the absence of eigenvalues below \(\eta\), can easily be ensured by Assumption~\ref{as:smd}. For \(\eta\) just below the critical scale of \(n^{-3/4}\), however, we need to prove an additional bound on the number of eigenvalues, as stated below. 
\begin{proposition}\label{prop:lt}
  For any \(n^{-1}\le\eta\le n^{-3/4}\) and \(\abs[\big]{\abs{z}^2-1}\lesssim n^{-1/2}\) we have the bound
  \begin{equation}
    \label{eq:impsmallbound}
    \begin{split}
      \E\abs{\set{i \given \abs{\lambda_i^z}\le \eta}}&\lesssim
      \begin{cases}
        n^{3/2} \eta^2 (1+\abs{\log(n\eta^{4/3})}), &\text{\(X\) complex} \\
        n^{3/4} \eta, &\text{\(X\) real}
      \end{cases}\\
      &\qquad +\landauO{\frac{n^\xi}{n^{5/2} \eta^3}},
    \end{split}
  \end{equation}
  on the number of small eigenvalues, for any \(\xi>0\).
\end{proposition}
We remark that the precise asymptotics of~\eqref{eq:impsmallbound} are of no importance for the proof of Theorem~\ref{theo:edgeuniv}. Instead it would be sufficient to establish that for any \(\epsilon>0\) there exists \(\delta>0\) such that we have \(\E\abs{\set{i \given \abs{\lambda_i^z}\le n^{-3/4-\epsilon}}}\lesssim n^{-\delta}\).

The paper is organized as follows: In Section~\ref{sec:univlt} we will prove Proposition~\ref{prop:lt} by a Green function comparison argument, using the analogous bound for the Gaussian case, as recently obtained in~\cite{1908.01653}. In Section~\ref{sec univ} we will then present the proof of our main result, Theorem~\ref{theo:edgeuniv}, which follows from combining the local law~\eqref{eq local law}, Girko's Hermitization identity~\eqref{eq girko}, the bound on small singular values~\eqref{eq:impsmallbound} and another long-time Green function comparison argument.

\section{Estimate on the lower tail of the smallest singular value of \texorpdfstring{\(X-z\)}{X-z}}\label{sec:univlt} 
The main result of this section is an estimate of the lower tail of the density of the smallest \(\abs{\lambda_i^z}\) in Proposition~\ref{prop:lt}.  For this purpose we introduce the following flow
\begin{equation}
  \label{eq:flowkeep12mom}
  \diff X_t=-\frac{1}{2} X_t \diff t+\frac{\diff  B_t}{\sqrt{n}},
\end{equation}
with initial data \(X_0=X\), where \(B_t\) is the real or complex matrix valued standard Brownian motion, i.e.\ \(B_t\in\R^{n\times n}\) or \(B_t\in\C  ^{n\times n}\), accordingly with \(X\) being real or complex, where \((b_t)_{ab}\) in the real case, and \(\sqrt{2}\Re[(b_t)_{ab}], \sqrt{2}\Im[(b_t)_{ab}]\) in the complex case, are independent standard real Brownian motions for \(a,b\in[n]\). The flow~\eqref{eq:flowkeep12mom} induces a flow \(\diff\chi_t=-\chi_t\diff t/2+\diff b_t\) on the entry distribution \(\chi\) with solution
\begin{equation}\label{chi OU}
  \chi_t = e^{-t/2}\chi + \int_0^t e^{-(t-s)/2}\diff b_s,\quad\text{i.e.}\quad \chi_t\stackrel{d}{=}e^{-t/2}\chi+ \sqrt{1-e^{-t}}g,
\end{equation}
where \(g\sim \mathcal{N}(0,1)\) is a standard real or complex Gaussian, independent of \(\chi\), with \(\E g^2=0\) in the complex case. By linearity of cumulants we find 
\begin{equation} \label{kappa cum}
  \kappa_{i,j}(\chi_t)=e^{-(i+j)t/2}\kappa_{i,j}(\chi) + \begin{cases} (1-e^{-t}) \kappa_{i,j}(g), & i+j=2\\ 0, &\text{else},
  \end{cases}
\end{equation}  
where \(\kappa_{i,j}(x)\) denotes the joint cumulant of \(i\) copies of \(x\) and \(j\) copies of \(\overline{x}\), in particular \(\kappa_{2,0}(x)=\kappa_{0,2}(x)=\kappa_{1,1}(x)=1\) for \(x=\chi,g\) in the real case, and \(\kappa_{0,2}(x)=\kappa_{2,0}(x)=0\ne\kappa_{1,1}(x)=1\) for \(x=\chi,g\) in the complex case. 

Thus~\eqref{chi OU} implies that, in distribution,
\begin{equation}
  \label{eq:disrtX}
  X_t\stackrel{d}{=} e^{-t/2} X_0+\sqrt{1-e^{-t}} \widetilde{X},
\end{equation}
where \(\widetilde{X}\) is a real or complex Ginibre matrix independent of \(X_0=X\). Then, we define the \(2n\times 2n\) matrix \(H_t=H_t^z\) as in~\eqref{eq:linmatrix} replacing \(X\) by \(X_t\), and its resolvent \(G_t(w)=G_t^z(w):= (H_t-w)^{-1}\), for any \(w\in \HC\). 
We remark that we defined the flow in~\eqref{eq:flowkeep12mom} with initial data \(X\) and not \(H^z\) in order to preserve the shape of the self consistent density of states of the matrix \(H_t\) along the flow. In particular, by~\eqref{eq:flowkeep12mom} it follows that \(H_t\) is the solution of the flow
\begin{equation}
  \label{eq:defbigflow}
  \diff H_t=-\frac{1}{2}(H_t+Z)\diff t+\frac{\diff \mathfrak{B}_t}{\sqrt{n}},\quad H_0=H=H^z
\end{equation}
with
\[
Z:= \left( \begin{matrix}
  0 & zI \\
  \overline{z}I & 0
\end{matrix}\right),
\qquad
\mathfrak{B}_t:= \left( \begin{matrix}
  0 & B_t \\
  B_t^* & 0
\end{matrix}\right),
\]
where \(I\) denotes the \(n \times n\) identity matrix.

\begin{proposition}\label{prop:cuspuniv}
  Let \(R_t:= \braket{ G_t(\ii \eta) } =\ii\braket{\Im G_t(\ii\eta)}\), then for any \(n^{-1}\le\eta\le n^{-3/4}\) it holds that
  \begin{equation}
    \label{eq:GFT}
    \abs{ \E  [R_{t_2}-R_{t_1}] } \lesssim \frac{(e^{-3t_1/2}-e^{-3t_2/2}) n^\xi}{n^{7/2}\eta^4},
  \end{equation}
  for any arbitrary small \(\xi>0\) and any \(0\le t_1< t_2\le +\infty\), with the convention that \(e^{-\infty}=0\).
\end{proposition}
\begin{proof}
  Denote \(W_t:=H_t+Z\). By~\eqref{eq:defbigflow} and Ito's Lemma it follows that 
  \begin{equation}
    \label{eq:ito}
    \E \frac{\diff R_t}{\diff t}=\E \left[-\frac{1}{2}\sum_\alpha w_\alpha(t) \partial_\alpha R_t+\frac{1}{2}\sum_{\alpha,\beta}\kappa_t(\alpha,\beta) \partial_\alpha\partial_\beta R_t\right],
  \end{equation}
  where \(\alpha, \beta\in[2n]^2\) are double indices, \(w_\alpha(t)\) are the entries of \(W_t\) and 
  \begin{equation}
    \kappa_t(\alpha,\beta,,\dots):=\kappa(w_\alpha(t), w_\beta(t),\dots)\label{kappa t def}
  \end{equation}
  denotes the joint cumulant of \(w_\alpha,w_\beta,\dots\), and \(\partial_\alpha:= \partial_{w_\alpha}\). By~\eqref{kappa cum} and 
  the independence  of \(\chi\) and \(g\) it follows that \(\kappa_t(\alpha,\beta)=\kappa_0(\alpha,\beta)\) for all \(\alpha,\beta\) and 
  \begin{align}
    \label{eq:cumulant}
    &\kappa_t(\alpha, \beta_1, \dots, \beta_j)\\
    &=\begin{cases}
      e^{-t\frac{j+1}{2}} n^{-\frac{j+1}{2}} \kappa_{l,k}(\chi) &\text{if }\alpha\not\in [n]^2\cup [n+1,2n]^2,\;\beta_i\in\{\alpha,\alpha'\} \, \forall  i\in[j]\\
      0 &\text{otherwise},
    \end{cases}\nonumber
  \end{align}
  for \(j>1\), where for a double index \(\alpha=(a,b)\), we use the notation \(\alpha':= (b,a)\), and \(l,k\) with \(l+k=j+1\) denote the number of double indices among \(\alpha,\beta_1,\dots,\beta_j\) which correspond to the upper-right, or respectively lower-left corner of the matrix \(H\). In the sequel the value of \(\kappa_{k,l}(\chi)\) is of no importance, but we note that Assumption~\ref{as:bm} ensures the bound \(\abs{\kappa_{k,l}(\chi)} \lesssim \sum_{j\le k+l} C_j<\infty\) for any \(k,l\), with \(C_j\) being the constants from Assumption~\ref{as:bm}. 
  
  We will use the cumulant expansion that holds for any smooth function \(f\):
  \begin{equation}\label{gen cum exp}
    \E w_\alpha f(w) = \sum_{m=0}^{K} \sum_{\beta_1,\dots,\beta_m\in [2n]^2} \frac{\kappa(\alpha,\beta_1,\dots,\beta_m)}{m!} \E \partial_{\beta_1}\dots\partial_{\beta_m} f(w) + \Omega(K,f),
  \end{equation}
  where the error term \(\Omega(K,f)\) goes to zero as the expansion order \(K\) goes to infinity. In our application the error is negligible for, say, \(K=100\) since with each derivative we gain an additional factor of \(n^{-1/2}\) and due to the independence~\eqref{eq:cumulant} the sums of any order have effectively only \(n^2\) terms.
  Applying~\eqref{gen cum exp} to~\eqref{eq:ito} with \(f=\partial_\alpha R_t\), the first order term is zero due to the assumption \(\E x_\alpha=0\), and the second order term cancels. The third order term is given by
  \begin{equation}
    \label{eq:nextordterm}
    \abs[\bigg]{\sum_{\alpha\beta_1\beta_2}\kappa_t(\alpha,\beta_1,\beta_2)\E [\partial_\alpha\partial_{\beta_1}\partial_{\beta_2} R_t]}\lesssim e^{-3t/2}\frac{n^\xi}{n^{7/2}\eta^4}.
  \end{equation}
  \begin{proof}[Proof of~\eqref{eq:nextordterm}]
    It follows from the resolvent identity that \(\partial_\alpha G=-G\Delta^\alpha G\), where \(\Delta^\alpha\) is the matrix of all zeros except for a \(1\) 
    in the \(\alpha\)-th entry\footnote{The matrix \(\Delta^\alpha\) is not to be confused with the Laplacian \(\Delta f\)  in Girko's formula~\eqref{eq girko}}. Thus, neglecting minuses and irrelevant constant factors, for any fixed \(\alpha\), the sum~\eqref{eq:nextordterm} is given by a sum of terms of the form
    \[ \braket{G_t \Delta^{\gamma_1} G_t\Delta^{\gamma_2} G_t\Delta^{\gamma_3} G_t }, \qquad \gamma_1,\gamma_2,\gamma_3\in\{\alpha,\alpha'\}.\]
    Hence, considering all possible choices of \(\gamma_{1},\gamma_2,\gamma_3\) and using independence to conclude that \(\kappa_t(\alpha,\beta_1,\beta_2)\) can only be non-zero if \(\beta_1,\beta_2\in\{\alpha,\alpha'\}\) we arrive at 
    \begin{align}\label{eq:firstexp}
      &\abs[\bigg]{\sum_{\alpha\beta_1\beta_2}\kappa_t(\alpha,\beta_1,\beta_2)\E [\partial_\alpha\partial_{\beta_1}\partial_{\beta_2} R_t]} \\\nonumber{}
      &\lesssim e^{-3t/2} n^{-5/2} \biggl(\abs[\bigg]{\sum_{abc}\Im \E G_{ca} G_{ba} G_{ba} G_{bc} } +\abs[\bigg]{\sum_{abc}\Im \E G_{ca} G_{ba} G_{bb} G_{ac} }\\
      &\qquad\qquad\qquad\qquad+\abs[\bigg]{\sum_{abc}\Im \E G_{ca} G_{bb} G_{aa} G_{bc} }\biggr),\nonumber
    \end{align}
    where the sums are taken over \((a,b)\in[2n]^2\setminus([n]^2\cup[n+1,2n]^2)\) and \(c\in[2n]\), and we dropped the time dependence of \(G=G_t\) for notational convenience. 
    
    We estimate the three sums in~\eqref{eq:firstexp} using that, by~\eqref{m scaling},~\eqref{M Mp bounds}, it follows
    \[ \abs{G_{ab}}\lesssim n^\xi,\qquad \abs{G_{aa}}\le \Im \wh m+\abs{(G-M)_{aa}}\lesssim n^{-1/4}+ 
    \eta^{1/3}+ \frac{n^\xi}{n\eta}\lesssim \frac{n^\xi}{n\eta}, \]
    from Proposition~\ref{prop:local law}, and Cauchy-Schwarz estimates by
    \[ 
    \begin{split}
      \sum_{abc} \abs{G_{ca}G_{ba}G_{ba}G_{bc}}&\le \sum_{ab} \abs{G_{ba}}^2 \sqrt{\sum_c \abs{G_{ca}}^2}\sqrt{\sum_{c} \abs{G_{bc}}^2 }\\
      &=\sum_{ab} \abs{G_{ba}}^2 \sqrt{(G^\ast G)_{aa}}\sqrt{(GG^\ast)_{bb} } \\
      &= \frac{1}{\eta}\sum_{ab} \abs{G_{ba}}^2 \sqrt{(\Im G)_{aa}}\sqrt{(\Im G)_{bb} } \lesssim \frac{n^{\xi}}{n\eta^2} \sum_b (GG^\ast)_{bb} \\
      &= \frac{n^{\xi}}{n\eta^3} \sum_b (\Im G)_{bb} \lesssim \frac{n^{2\xi}}{n\eta^4}, 
    \end{split} \]
    and similarly 
    \[ 
    \begin{split}
      \sum_{abc} \abs{G_{ca}G_{ba}G_{bb}G_{ac}}&\lesssim \frac{n^{\xi}}{n\eta^2} \sum_{ab} \abs{G_{ba}} (\Im G)_{aa} \\
      &\le \frac{n^{\xi}}{n^{1/2}\eta^{5/2}} \sum_{a}  (\Im G)_{aa} \sqrt{(\Im G)_{aa}} \lesssim \frac{n^{5\xi/2}}{n\eta^4} 
    \end{split} \]
    and 
    \[ 
    \begin{split}
      \sum_{abc} \abs{G_{ca}G_{bb}G_{aa}G_{bc}}&\lesssim \frac{n^{2\xi}}{n^2\eta^3} \sum_{ab} \sqrt{(\Im G)_{aa}}\sqrt{(\Im G)_{bb}} \lesssim \frac{n^{3\xi}}{n\eta^4}.
    \end{split} \]
    This concludes the proof of~\eqref{eq:nextordterm} by choosing \(\xi\) in Proposition~\ref{prop:local law} accordingly.
    
  \end{proof}

  Finally, in the cumulant expansion of~\eqref{eq:ito} we are able to bound the terms of order at least four trivially. Indeed, for the fourth order, the trivial bound is \(e^{-2t}\) since the \(n^3\) from the summation is compensated by the \(n^{-2}\) from the cumulants and the \(n^{-1}\) from the normalization of the trace. Morever, we can always perform at least two Ward-estimates on the first and last \(G\) with respect to the trace index. Thus we can estimate any fourth-order term by \( e^{-2t}(n\eta)^{-2}\le e^{-3t/2}n^{-7/2}\eta^{-4}\), and we note that the power-counting for higher order terms is even better than that. Whence we have shown that \(\E \abs{\diff R_t/\diff t} \lesssim e^{-3t/2}n^{-7/2}\eta^{-4} \) and the proof of Proposition~\ref{prop:cuspuniv}
  is complete after integrating~\eqref{eq:ito} in \(t\) from \(t_1\) to \(t_2\).
\end{proof}

Let \(\widetilde{X}\) be a real or complex \(n\times n\) Ginibre matrix and let \(\widetilde{H}^z\) be the linearized matrix defined as in~\eqref{eq:linmatrix} replacing \(X\) by \(\widetilde{X}\). Let \(\widetilde\lambda_i=\widetilde\lambda_i^z\), with \(i\in\{1,\dots, 2n\}\), be the eigenvalues of \(\widetilde{H}^z\). We define the non negative Hermitian matrix \(\widetilde{Y}=\widetilde{Y}^z:= (\widetilde{X}-z)(\widetilde{X}-z)^*\), then, by~\cite[Eq. (13c)-(14)]{1908.01653} it follows that for any \(\eta\le n^{-3/4}\) we have
\begin{equation}
  \label{eq:boundsmalleiggau}
  \E \Tr  \big[\widetilde{Y}+\eta^2\big]^{-1}= \E\sum_{i=1}^{2n} \frac{1}{\widetilde\lambda_i^2+\eta^2}\lesssim
  \begin{cases}
    n^{3/2} (1+\abs{\log(n\eta^{4/3})}), &\Gin(\C), \\
    n^{3/4} \eta^{-1}, &\Gin(\R),
  \end{cases}
\end{equation}
for \(\wt X\) distributed according to the complex, or respective, real Ginibre ensemble.

Combining~\eqref{eq:boundsmalleiggau} and Proposition~\ref{prop:cuspuniv} we now present the proof of Proposition~\ref{prop:lt}.
\begin{proof}[Proof of Proposition~\ref{prop:lt}]
  Let \(\lambda_i(t)\), with \(i\in\{1,\dots, 2n\}\), be the eigenvalues of \(H_t\) for any \(t\ge 0\). Note that \(\lambda_i(0)=\lambda_i\), since \(H_0=H^z\). By~\eqref{eq:GFT}, choosing \(t_1=0\), \(t_2=+\infty\) it follows that
  \begin{equation}
    \label{eq:estsmalleigs}
    \begin{split}
      \E_{H_t} \abs{\set{i\given \abs{\lambda_i}\le \eta}} &\le \eta\cdot \E_{H_t}\left(\Im \sum_{i=1}^{2n} \frac{1}{\lambda_i-\ii \eta}\right) \\
      &=\eta^2 \cdot \E _{H_\infty}\left(\sum_{i=1}^{2n} \frac{1}{\lambda_i^2+\eta^2} \right) +\mathcal{O}\left(\frac{n^\xi}{n^{5/2} \eta^3}\right),
    \end{split}
  \end{equation}
  for any \(\xi>0\). Since the distribution of \(H_\infty\) is the same as \(\widetilde{H}^z\) it follows that
  \[
  \E _{\widetilde{H}^z}\left(\sum_{i=1}^{2n} \frac{1}{\mu_i^2+\eta^2} \right)=2\E _{\widetilde{X}} \Tr\big[\widetilde{Y}+\eta^2\big]^{-1},
  \]
  and combining~\eqref{eq:boundsmalleiggau} with~\eqref{eq:estsmalleigs}, we immediately conclude the bound in~\eqref{eq:impsmallbound}.
\end{proof}

\section{Edge universality for non-Hermitian random matrices}\label{sec univ}
In this section we prove our main edge universality result, as stated in Theorem~\ref{theo:edgeuniv}.

In the following of this section without loss of generality we can assume that the test function \(F\) is of the form
\begin{equation}\label{Fprod}
  F(w_1,\dots, w_k)=f^{(1)}(w_1)\cdots f^{(k)}(w_k),
\end{equation}
with \(f^{(1)},\dots, f^{(k)}\colon\C  \to \C  \) being smooth and compactly supported functions. Indeed, any smooth function \(F\) can 
be effectively approximated by its truncated Fourier series (multiplied by smooth cutoff function of product form); see also~\cite[Remark 3]{MR3306005}. Using the 
effective decay of the Fourier coefficients of \(F\) controlled by its \(C^{2k+1}\) norm, a standard approximation argument 
shows that if~\eqref{eq:univ} holds for \(F\) in the product form~\eqref{Fprod} with an error \(\landauO{ n^{-c(k)}}\), then 
it also holds for a general smooth function with an error \(\landauO{ n^{-c}}\), where the implicit constant in \(\mathcal{O}(\cdot)\) depends on \(k\) and on the \(C^{2k+1}\)-norm of \(F\), and the constant \(c>0\) depends on \(k\).

To resolve eigenvalues on their natural scale we consider the rescaling \(f_{z_0}(z):= n f(\sqrt{n}(z-z_0))\) and compare the linear statistics 
\(n^{-1}\sum_i f_{z_0}(\sigma_i)\) and \(n^{-1}\sum_i f_{z_0}(\wt\sigma_i)\), with \(\sigma_i,\wt\sigma_i\) being the eigenvalues of \(X\) and of the comparison Ginibre ensemble \(\wt X\), respectively. For convenience we may normalize both linear statistics by their deterministic approximation from the local law~\eqref{eq local law} which, according to~\eqref{eq girko} is given by 
\begin{equation} \label{eq:dos} 
  \frac{1}{n}\sum_{i} f_{z_0}(\sigma_i)\approx \frac{1}{\pi}\int_\DD f_{z_0}(z)\diff z,
\end{equation}
where \(\DD\) denotes the unit disk of the complex plane.
\begin{proposition}\label{prop:mainprop}
  Let \(k\in\N\) and  \(z_1, \dots, z_k\in\C\) be such that \(\abs{z_j}^2=1\) for all \(j\in[k]\), and let \(f^{(1)},\dots,f^{(k)}\) be smooth compactly supported test functions. Denote the eigenvalues of an i.i.d.\ matrix \(X\) satisfying Assumptions~\ref{as:bm}--\ref{as:smd} and a corresponding real or complex Ginibre matrix \(\widetilde{X}\) by \(\{\sigma_i\}_{i=1}^n\), \(\{\widetilde{\sigma}_i\}_{i=1}^n\). Then we have the bound
  \begin{equation}\label{eq:diffbound}
    \begin{split}
      &\E \Biggl[ \prod_{j=1}^k\left(\frac{1}{n}\sum_{i=1}^n f^{(j)}_{z_j}(\sigma_i)-\frac{1}{\pi} \int_\DD  f^{(j)}_{z_j}(z) \diff z \right)\\
      &\qquad\qquad- \prod_{j=1}^k\left(\frac{1}{n}\sum_{i=1}^n f^{(j)}_{z_j}(\widetilde{\sigma}_i)- \frac{1}{\pi}\int_\DD  f^{(j)}_{z_j}(z) \diff z \right) \Biggr]= \landauO { n^{-c(k)}},
    \end{split} 
  \end{equation}
  for some small constant \(c(k)>0\), where the implicit  multiplicative constant 
  in  \(\mathcal{O}(\cdot)\) depends on the norms \(\norm{\Delta f^{(j)}}_1\), \(j=1,2, \ldots , k\).
\end{proposition}
\begin{proof}[Proof of Theorem~\ref{theo:edgeuniv}]
  Theorem~\ref{theo:edgeuniv} follows directly from Proposition~\ref{prop:mainprop} by the definition of the \(k\)-point correlation function in~\eqref{eq:kpointfunc}, the exclusion-inclusion principle and the bound
  \[ \abs*{ \frac{1}{\pi}\int_\DD   f_{z_0}(z) \diff z} \lesssim 1.\qedhere\] 
\end{proof}

The remainder of this section is devoted to the proof of Proposition~\ref{prop:mainprop}. We now fix some \(k\in\N\) and some \(z_1,\dots,z_k,f^{(1)},\dots,f^{(k)}\) as in Proposition~\ref{prop:mainprop}. All subsequent estimates in this section, also if not explicitly stated, hold true uniformly for any \(z\) in an order
\(n^{-1/2}\)-neighborhood of \(z_1,\dots, z_k\). In order to prove~\eqref{eq:diffbound}, we use Girko's formula~\eqref{eq girko} to write
\begin{equation}
  \label{eq:defGirform}
  \begin{split}
    &\frac{1}{n}\sum_{i=1}^n f^{(j)}_{z_j}(\sigma_i)- \frac{1}{\pi}\int_\DD  f^{(j)}_{z_j}(z) \diff z =I_1^{(j)}+I_2^{(j)}+I_3^{(j)}+I_4^{(j)},
  \end{split}
\end{equation}
where 
\[\begin{split}
  I_1^{(j)}&:= \frac{1}{4\pi n}\int_\C   \Delta f^{(j)}_{z_j}(z)\log\abs{\det (H^z-\ii T)} \diff z \\
  I_2^{(j)}&:=-\frac{1}{2\pi}\int_\C   \Delta f^{(j)}_{z_j}(z)\int_0^{\eta_0} \left[\braket{\Im G^z(\ii \eta)}-\Im \widehat{m}^z(\ii \eta)\right]\, \diff \eta\diff z \\
  I_3^{(j)}&:= -\frac{1}{2\pi}\int_\C   \Delta f^{(j)}_{z_j}(z)\int_{\eta_0}^{T} \left[\braket{\Im G^z(\ii \eta)}-\Im \widehat{m}^z(\ii \eta)\right]\, \diff \eta\diff z \\
  I_4^{(j)}&:= +\frac{1}{2\pi}\int_\C   \Delta f^{(j)}_{z_j}(z)\int_T^{+\infty} \left( \Im \widehat{m}^z(i\eta)-\frac{1}{\eta+1}\right) \diff \eta\diff z ,
\end{split} \]
with \(\eta_0:= n^{-3/4-\delta}\), for some small fixed \(\delta>0\), and for some very large \(T>0\), say \(T:= n^{100}\). We define \(\widetilde{I}_1^{(j)}\), \(\widetilde{I}_2^{(j)}\), \(\widetilde{I}_3^{(j)}\), \(\widetilde{I}_4^{(j)}\) analogously for the Ginibre ensemble by replacing \(H^z\) by \(\wt H^z\) and \(G^z\) by \(\wt G^z\).

\begin{proof}[Proof of Proposition~\ref{prop:mainprop}]
  The first step in the proof of Proposition~\ref{prop:mainprop} is the reduction to a corresponding statement about the \(I_3\)-part in~\eqref{eq:defGirform}, as summarized in the following lemma.
  \begin{lemma}\label{lem:mainredux}
    Let \(k\ge 1\), let \(I_3^{(1)},\dots, I_3^{(k)}\) be the integrals defined in~\eqref{eq:defGirform}, with \(\eta_0=n^{-3/4-\delta}\), for some small fixed \(\delta>0\), and let \(\widetilde{I}_3^{(1)}, \dots, \widetilde{I}_3^{(k)}\) be defined as in~\eqref{eq:defGirform} replacing \(m^z\) with \(\widetilde{m}^z\). Then,
    \begin{equation}\label{eq:easier}
      \begin{split}
        &\E \left[ \prod_{j=1}^k\left(\frac{1}{n}\sum_{i=1}^n f_{z_j}^{(j)}(\sigma_i)- \frac{1}{\pi}\int_\DD  f_{z_j}^{(j)}(z) \diff z \right)- \prod_{j=1}^k\left(\frac{1}{n}\sum_{i=1}^n f_{z_j}^{(j)}(\widetilde{\sigma}_i)- \frac{1}{\pi}\int_\DD  f_{z_j}^{(j)}(z) \diff z \right) \right] \\
        &\qquad= \E \left[ \prod_{j=1}^k I_3^{(j)}-\prod_{j=1}^k \widetilde{I}_3^{(j)}\right]+ \mathcal{O}\left( n^{-c_2(k,\delta)}\right),
      \end{split}
    \end{equation}
    for some small constant \(c_2(k,\delta)>0\).
  \end{lemma} 
  In order to conclude the proof of Proposition~\ref{prop:mainprop}, due to Lemma~\ref{lem:mainredux}, it only remains to prove that
  \begin{equation}
    \label{eq:finbouprod}
    \E \left[ \prod_{j=1}^k I_3^{(j)}-\prod_{j=1}^k \widetilde{I}_3^{(j)}\right]=\mathcal{O}\left( n^{-c(k)}\right),
  \end{equation}
  for any fixed \(k\) with some small constant \(c(k)>0\), where we recall the definition of \(I_3\) and the corresponding \(\widetilde I_3\) for Ginibre from~\eqref{eq:defGirform}. The proof of~\eqref{eq:finbouprod} is similar to the Green function comparison proof in Proposition~\ref{prop:cuspuniv} but more involved due to the fact that we compare products of resolvents and that we have an additional \(\eta\)-integration. Here we define the observable
  \begin{equation}\label{eq:impobs}
    Z_t := \prod_{j\in[k]} I_3^{(j)}(t):=\prod_{j\in[k]} \biggl(-\frac{1}{2\pi}\int_\C   \Delta f^{(j)}_{z_j}(z)\int_{\eta_0}^T \Im\braket{G^z_t(\ii\eta)-M^z(\ii\eta)}\diff \eta\diff z\biggr),
  \end{equation}
  where we recall that \(G_t^z(w):= (H^z_t-w)^{-1}\) with \(H^z_t=H_t\) as in~\eqref{eq:defbigflow}.
  \begin{lemma}\label{prop:impnewflow}
    For any \(n^{-1}\le\eta_0\le n^{-3/4}\) and \(T=n^{100}\) and any small \(\xi>0\) it holds that
    \begin{equation}
      \label{eq:fingcr}
      \abs{\E  [Z_{t_2}- Z_{t_1}]}\lesssim  \Bigl(e^{-3t_0/2}-e^{-3t_1/2}\Bigr) \frac{n^\xi}{n^{5/2}\eta_0^3}\prod_j \norm[\big]{\Delta f^{(j)}}_1
    \end{equation}
    uniformly in \(0\le t_1<t_2\le +\infty\) with the convention that \(e^{-\infty}=0\). 
  \end{lemma}
  Since \(Z_0=\prod_j I_3^{(j)}\) and \(Z_\infty=\prod_j \wt I_3^{(j)}\), the proof of Proposition~\ref{prop:mainprop} follows directly from~\eqref{eq:finbouprod}, modulo the proofs of Lemmata~\ref{lem:mainredux}--\ref{prop:impnewflow} that will be given in the next two subsections.
\end{proof}

\subsection{Proof of Lemma~\ref{lem:mainredux}}
In order to estimate the probability that there exists an eigenvalue of \(H^z\) very close to zero, we use the following proposition that has been proven in~\cite[Prop. 5.7]{MR3770875} adapting the proof of~\cite[Lemma 4.12]{MR2908617}.

\begin{proposition}\label{prop:verysmalleig}
  Under Assumption~\ref{as:smd} there exists a constant \(C>0\), depending only on \(\alpha\), such that
  \begin{equation}
    \label{eq:estvers}
    \Prob\left(\min_{i\in[2n]}\abs{\lambda_i^z}\le \frac{u}{n} \right)\le C u^\frac{2\alpha}{1+\alpha}n^{\beta+1},
  \end{equation}
  for all \(u>0\) and \(z\in \C  \).
\end{proposition}

In the following lemma we prove a very high probability bound for \(I_1^{(j)}\), \(I_2^{(j)}\), \(I_3^{(j)}\), \(I_4^{(j)}\). The same bounds hold true for \(\widetilde{I}_1^{(j)}\), \(\widetilde{I}_2^{(j)}\), \(\widetilde{I}_3^{(j)}\), \(\widetilde{I}_4^{(j)}\) as well. These bounds in the bulk regime were already proven in~\cite[Proof of Theorem 2.5]{MR3770875}, the current edge regime is analogous, 
so we only provide a sketch of the proof for completeness.

\begin{lemma}\label{lem:aprioribound}
  For any \(j\in[k]\) the bounds
  \begin{equation}
    \label{eq:higprobbound}
    \abs{I_1^{(j)}}\le \frac{n^{1+\xi}\norm{\Delta f^{(j)}}_1}{T^2}, \quad \abs{I_2^{(j)}}+ \abs{I_3^{(j)}}\le n^\xi\norm[\big]{\Delta f^{(j)}}_1 , \quad \abs{I_4^{(j)}}\le \frac{n\norm{\Delta f^{(j)}}_1}{T},
  \end{equation}
  hold with very high probability for any \(\xi>0\). The bounds analogous to~\eqref{eq:higprobbound} also hold for \(\wt I_l^{(j)}\).
\end{lemma}
\begin{proof}
  For notational convenience we do not carry the \(j\)-dependence of \(I_l^{(j)}\) and \(f^{(j)}\), and the dependence of \(\lambda_i,H,G,M,\widehat m\) on \(z\) within this proof. Using that 
  \[
  \log \abs{\det (H -\ii T)}= 2n\log T+\sum_{j\in[n]}\log \left( 1+\frac{\lambda_j^{2}}{T^2}\right),
  \]
  we easily estimate \(\abs{I_1}\) as follows
  \[
  \begin{split}
    \abs{I_1} &=\abs*{ \frac{1}{4\pi n}\int_\C   \Delta f_{z_j}(z)\log\abs{\det (H-\ii T)} \diff z }\\&\lesssim \frac{1}{n}\int_\C   \abs{\Delta f_{z_j}(z)}\frac{\Tr  H^2}{T^2} \diff z\lesssim \frac{n^{1+\xi}\norm{\Delta f}_1}{T^2},
  \end{split}
  \]
  for any \(\xi>0\) with very high probability owing to the high moment bound~\eqref{eq:boundmom}. 
  By~\eqref{eq:defhatm} it follows that \(\abs{\Im \wh m^z(\ii\eta)-(\eta+1)^{-1}}\sim \eta^{-2}\) for large \(\eta\), proving also the bound on \(I_4\) in~\eqref{eq:higprobbound}. The bound for \(I_3\) follows immediately from the averaged local law in~\eqref{eq local law}. 
  
  For the \(I_2\) estimate we split the \(\eta\)-integral of \(\Im m^z(\ii \eta)-\Im \widehat{m}^z(\ii\eta)\) in \(I_2\) as follows
  \begin{align}
    \label{eq:split}
    &  \int_0^{\eta_0} \Im \braket{G^z(\ii\eta)-M^z(\ii\eta)} \diff \eta \\
    &\qquad =\frac{1}{n}\sum_{\abs{\lambda_i}< n^{-l}}\log\left( 1+\frac{\eta_0^2}{\lambda_i^2}\right) +\frac{1}{n} \sum_{\abs{\lambda_i}\ge n^{-l}}\log\left( 1+\frac{\eta_0^2}{\lambda_i^2}\right)-\int_0^{\eta_0} \Im \widehat{m}^z(\ii\eta) \diff \eta,\nonumber
  \end{align}
  where \(l\in \N\) is a large fixed integer. Using~\eqref{m scaling} we find that the third term in~\eqref{eq:split} is bounded by \(n^{-1-\delta}\). Choosing \(l\) large enough, it follows, as in~\cite[Eq.~(5.35)]{MR3770875}, using the bound~\eqref{eq:estvers} that
  \begin{equation}
    \label{eq:hopgd}
    \frac{1}{n}\sum_{\abs{\lambda_i}< n^{-l}}\log\left( 1+\frac{\eta_0^2}{\lambda_i^2}\right)\le n^{-1+\xi},
  \end{equation}
  with very high probability for any \(\xi>0\). Alternatively, this bound also follows from~\eqref{eq:tv}
  without Assumption~\ref{as:smd}, circumventing Proposition~\ref{prop:verysmalleig}, see Remark~\ref{remm:tv}.
  For the second term in~\eqref{eq:split} we define \(\eta_1:= n^{-3/4+\xi}\) with some very small \(\xi>0\) and using \(\log (1+x)\le x\) we write
  \begin{equation}\label{eq:usefsplit}
    \begin{split}
      \sum_{\abs{\lambda_i}\ge n^{-l}}\log\left( 1+\frac{\eta_0^2}{\lambda_i^2}\right)&=  \sum_{ n^{-l}\le \abs{\lambda_i}\le n^{\delta/2} \eta_0}\log\left( 1+\frac{\eta_0^2}{\lambda_i^2}\right)+
      \eta_0^2 \sum_{\abs{\lambda_i}\ge n^{\delta/2}\eta_0} \frac{1}{\lambda_i^2} \\
      &\lesssim \abs{\set{i\given\abs{\lambda_i}< n^{\delta/2}\eta_0}} \cdot \log n + \eta_0^2 \sum_{\abs{\lambda_i}\ge n^{\delta/2}\eta_0} \frac{1}{\lambda_i^2} \\
      &\lesssim (\log n) n^{4\xi/3} + \frac{\eta_0^2 n^{\delta + 2\xi}}{\eta_1}\sum_{\abs{\lambda_i}\ge n^{\delta/2}\eta_0} 
      \frac{\eta_1}{\lambda_i^2+\eta_1^2} \\
      &\lesssim  (\log n) n^{4\xi/3}+ n^{1-\delta}\eta_1 \braket{\Im G^z(\ii\eta_1)} \le    n^{2\xi} + n^{-\delta+2\xi}
    \end{split}
  \end{equation}
  by the averaged local law in~\eqref{eq local law}, and \(\braket{\Im M^z(\ii \eta_1)} \lesssim  \eta_1^{1/3}\)
  from~\eqref{m scaling}.
  Here  from the second to third line in~\eqref{eq:usefsplit} we used that
  \begin{equation}
    \label{eq:nottgb}
    \abs{\set{i\given\abs{\lambda_i}\le n^{\delta/2}\eta_0}}\le\sum_i \frac{\eta_1^2}{\lambda_i^2 + \eta_1^2}  =n\eta_1 \braket{\Im G^z(\ii\eta_1)}\le n^{4\xi/3},
  \end{equation}
  again by the local law. By redefining \(\xi\), 
  this concludes the high probability bound on \(I_2\) in~\eqref{eq:higprobbound}, and thereby the proof of the lemma. 
\end{proof}
In the following lemma we prove an improved bound for \(I_2^{(j)}\), compared with~\eqref{eq:higprobbound}, which holds true only in expectation. The main input of the following lemma is the stronger lower tail estimate on \(\lambda_i\), in the regime \(\abs{\lambda_i}\ge n^{-l}\), from~\eqref{eq:impsmallbound} instead of~\eqref{eq:nottgb}.
\begin{lemma}\label{lem:averbound}
  Let \(I_2^{(j)}\) be defined in~\eqref{eq:defGirform}, then
  \begin{equation}
    \label{eq:carfbound}
    \E  \abs[\big]{I_2^{(j)}}\lesssim n^{-\delta/3} \lVert \Delta f^{(j)} \rVert_1, 
  \end{equation}
  for any \(j\in \{1,\dots, k\}\).
\end{lemma}
\begin{proof}
  We split the \(\eta\)-integral of \(\Im m^z(\ii \eta)-\Im \widehat{m}^z(\ii\eta)\) as in~\eqref{eq:split}. The third term in the r.h.s.\ of~\eqref{eq:split} is of order \(n^{-1-4\delta/3}\). Then, we estimate the first term in the r.h.s.\ of~\eqref{eq:split} as
  \begin{align}
    \label{eq:estsmeig1}
    \E  \left[ \frac{1}{n}\sum_{\abs{\lambda_i}< n^{-l}} \log\left( 1+\frac{\eta_0^2}{\lambda_i^2}\right)\right] &\le \E  \left[ \log\left( 1+\frac{\eta_0^2}{\lambda_1^2}\right) \1(\lambda_1\le n^{-l})\right]\\
    &\lesssim \E [\abs{\log \lambda_1} \1(\lambda_1\le n^{-l})] \nonumber\\
    &=\int_{l\log n}^{+\infty}\Prob(\lambda_1\le e^{-t}) \, \diff t\lesssim n^{\beta+1+\frac{2\alpha}{1+\alpha}} e^{-\frac{2\alpha l}{1+\alpha}},\nonumber
  \end{align}
  where in the last inequality we use~\eqref{eq:estvers} with \(u=e^{-t} n\). Note that by~\eqref{eq:impsmallbound} it follows that
  \begin{equation}
    \label{eq:needs}
    \E \abs[\big]{ \{i:\abs{\lambda_i}\le n^{\delta/2}\eta_0\} }\lesssim n^{-\delta/2}.
  \end{equation}
  Hence, by~\eqref{eq:needs}, using similar computations to~\eqref{eq:usefsplit}, we conclude that
  \begin{equation}
    \label{eq:estsmeig2}
    \E \left[\frac{1}{n} \sum_{\abs{\lambda_i}\ge n^{-l}}\log\left( 1+\frac{\eta_0^2}{\lambda_i^2}\right)\right]\lesssim \frac{\log n}{n^{1+\delta/2}}.
  \end{equation}
  Note that the only difference to prove~\eqref{eq:estsmeig2} respect to~\eqref{eq:usefsplit} is that the first term in the first line of the r.h.s.\ of~\eqref{eq:usefsplit} is estimated using~\eqref{eq:needs} instead of~\eqref{eq:nottgb}. Finally, choosing \(l\ge \alpha^{-1} (3+\beta)(1+\alpha)+2\), and combining~\eqref{eq:estsmeig1},~\eqref{eq:estsmeig2} we conclude~\eqref{eq:carfbound}. 
\end{proof}

Equipped with Lemmata~\ref{lem:aprioribound}--\ref{lem:averbound}, we now present the proof of Lemma~\ref{lem:mainredux}.
\begin{proof}[Proof of Lemma~\ref{lem:mainredux}]
  Using the definitions for \(I_1^{(j)}, I_2^{(j)},I_3^{(j)}, I_4^{(j)}\) in~\eqref{eq:defGirform}, and similar definitions for \(\widetilde{I}_1^{(j)}, \widetilde{I}_2^{(j)},\widetilde{I}_3^{(j)}, \widetilde{I}_4^{(j)}\), we conclude that
  \[
  \begin{split}
    &\E \left[ \prod_{j=1}^k\left(\frac{1}{n}\sum_{i=1}^n f_{z_j}^{(j)}(\sigma_i)- \frac{1}{\pi}\int_\DD  f_{z_j}^{(j)}(z) \diff z \right)- 
    \prod_{j=1}^k\left(\frac{1}{n}\sum_{i=1}^n f_{z_j}^{(j)}(\widetilde{\sigma}_i)- \frac{1}{\pi}\int_\DD  f_{z_j}^{(j)}(z) \diff z \right) \right] \\
    &\qquad =  \E \left[ \prod_{j=1}^k\left(I_1^{(j)}+ I_2^{(j)}+I_3^{(j)}+ I_4^{(j)}\right)- \prod_{j=1}^k\left(\widetilde{I}_1^{(j)}+ \widetilde{I}_2^{(j)}+\widetilde{I}_3^{(j)}+ \widetilde{I}_4^{(j)}\right)\right] \\
    &\qquad =  \E \left[ \prod_{j=1}^k I_3^{(j)}-\prod_{j=1}^k \widetilde{I}_3^{(j)}\right]+\sum_{\substack{j_1+j_2+j_3+j_4=k, \\ j_i\ge 0, \, j_3<k}} \E  \prod_{\substack{i_l=1, \\ l=1,2,3,4}}^{j_l} I_1^{(i_1)} I_2^{(i_2)} I_3^{(i_3)} I_4^{(i_4)} \\
    &\qquad\quad -\sum_{\substack{j_1+j_2+j_3+j_4=k, \\ j_i\ge 0, \, j_3<k}} \E  \prod_{\substack{i_l=1, \\ l=1,2,3,4}}^{j_l} \widetilde{I}_1^{(i_1)} \widetilde{I}_2^{(i_2)} \widetilde{I}_3^{(i_3)} \widetilde{I}_4^{(i_4)}.
  \end{split}
  \]
  Then, if \(j_2\ge 1\), by Lemma~\ref{lem:aprioribound} and Lemma~\ref{lem:averbound}, using that \(T=n^{100}\) in the definition of \(I_1^{(j)},\dots, I_4^{(j)}\) in~\eqref{eq:defGirform}, it follows that
  \[
  \E  \prod_{\substack{i_l=1, \\ l=1,2,3,4}}^{j_l} I_1^{(i_1)} I_2^{(i_2)} I_3^{(i_3)} I_4^{(i_4)}\lesssim \frac{n^{j_1+j_4} n^{(k-j_4-1)\xi} \ \prod_{j=1}^k \Vert\Delta f^{(j)}\rVert_1}{n^{\delta/3} T^{2j_1+j_4}}\le n^{-c_2(k,\delta)},
  \]
  for any \(j_1,j_3, j_4 \ge 0\), and a small constant \(c(_2k,\delta)>0\) which only depends on \(k, \delta\). If, instead, \(j_2=0\), then at least one among \(j_1\) and \(j_4\) is not zero, since \(0\le j_3\le k-1\) and \(j_1+j_2+j_3+j_4=k\). Assume \(j_1\ge 1\), the case \(j_4\ge 1\) is completely analogous, then
  \[
  \E  \prod_{\substack{i_l=1, \\ l=1,2,3,4}}^{j_l} I_1^{(i_1)} I_2^{(i_2)} I_3^{(i_3)} I_4^{(i_4)}\lesssim \frac{n^{j_1+j_4} n^{(k-j_4)\xi} 
  \prod_{j=1}^k \lVert \Delta f^{(j)}\rVert_1}{T^{2j_1+j_4}} \le n^{-c_2(k,\delta)}.
  \]
  Since similar bounds hold true for \(\widetilde{I}_1^{(i_1)}, \widetilde{I}_2^{(i_2)}, \widetilde{I}_3^{(i_3)}, \widetilde{I}_4^{(i_4)}\) as well, the above inequalities conclude the proof of~\eqref{eq:easier}.
\end{proof}

\subsection{Proof of Lemma~\ref{prop:impnewflow}}
We begin with a lemma generalizing the bound in~\eqref{eq:higprobbound} to derivatives of \(I_3^{(j)}\). 
\begin{lemma}\label{lemma fixed alpha}Assume \(n^{-1}\le \eta_0\le n^{-3/4}\) and fix \(l\ge 0\), \(j\in[k]\) and a double index  \(\alpha=(a,b)\) such that \(a\neq b\). Then, for any choice of \(\gamma_i\in\{\alpha,\alpha'\}\) and any \(\xi>0\) we have the bounds
  \begin{equation}\label{fixed alpha bounds}
    \abs{ \partial_\gamma^l I_3^{(j)}(t) }\lesssim \norm[\big]{\Delta f^{(j)}}_1  n^\xi \biggl( \frac{1}{(n\eta_0)^{\min\{l,2\}}} + \1\bigl(a\equiv b+n\pmod{2n}\bigr) \biggr),
  \end{equation}
  where \(\partial_\gamma^l:=\partial_{\gamma_1}\dots\partial_{\gamma_l}\), 
  with very high probability uniformly in \(t\ge 0\). 
\end{lemma}
\begin{proof} 
  We omit the \(t\)- and \(z\)-dependence of \(G_t^z\), \(\wh m^z\) within this proof since all bounds hold uniformly in \(t\ge 0\) and \(\abs{z-z_j}\lesssim n^{-1/2}\). 
  We also omit the \(\eta\)-argument from these functions, but the \(\eta\)-dependence of  all estimates will explicitly be indicated.
  Note that the \(l=0\) case was already proven in~\eqref{eq:higprobbound}. We now separately consider the remaining cases \(l=1\) and \(l\ge 2\). For notational simplicity we neglect the \(n^\xi\) multiplicative error factors (with arbitrarily small exponents \(\xi>0\)) applications of the local law~\eqref{eq local law} within the proof. In particular we will repeatedly use~\eqref{eq local law} in the form
  \begin{equation}\label{G local law bounds}
    \begin{split}
      \abs{G_{ba}} &\lesssim \begin{cases} 1, & a\equiv b+n\pmod{2n},\\ \psi, & a\not\equiv b+n \pmod{2n},\end{cases} \quad G_{bb}=\wh m + \landauO{\psi},\\
      \abs{\wh m}&\lesssim \min\{1,\eta^{1/3}+n^{-1/4}\},
    \end{split}  
  \end{equation}
  where we defined the parameter \[\psi:= \frac{1}{n\eta}+\frac{1}{n^{1/2}\eta^{1/3}}.\]
  \subsubsection*{Case \texorpdfstring{\(l=1\)}{l=1}} This follows directly from 
  \[ \begin{split}
    \abs*{\int_{\eta_0}^T \braket{G \Delta^{ab} G}\diff\eta}&= \abs{*\frac{1}{n} \int_{\eta_0}^T G^2_{ba}\diff\eta} = \frac{\abs{G(\ii T)_{ab}-G(\ii \eta_0)_{ab}}}{n}\\
    &\lesssim \frac{1}{n^2\eta_0} + \frac{1}{n}\1\bigl(a\equiv b+n\pmod{2n}\bigr),
  \end{split} \]
  where in the last step we used \(\norm{G(\ii T)}\le T^{-1}=n^{-100}\) and~\eqref{G local law bounds}. Since this bound is uniform in \(z\) we may bound the remaining integral by \( n\norm{\Delta f^{(j)}}_1 \), proving~\eqref{fixed alpha bounds}. 
  \subsubsection*{Case \texorpdfstring{\(l\ge 2\)}{l>=2}} For the case \(l\ge 2\) there are many assignments of \(\gamma_i\)'s to consider, e.g.
  \[\begin{split} 
    &\braket{G\Delta^{ab}G\Delta^{ab}G} = \frac{1}{n} \sum_c G_{ca} G_{ba} G_{bc},\quad \braket{G\Delta^{ab}G\Delta^{ba}G} = \frac{1}{n} \sum_c G_{ca} G_{bb} G_{ac}, \\ 
    &\braket{G\Delta^{ab}G\Delta^{ba}G\Delta^{ab}G} = \frac{1}{n} \sum_c G_{ca} G_{bb} G_{aa} G_{bc},\\ 
    &\braket{G\Delta^{ab}G\Delta^{ba}G\Delta^{ba}G} = \frac{1}{n} \sum_c G_{ca} G_{bb} G_{ab} G_{ac} 
  \end{split}\]
  but all are of the form that there are two \(G\)-factors carrying the independent summation index \(c\). In the case that \(a\equiv b+n\pmod{2n}\) we simply bound all remaining \(G\)-factors by \(1\) using~\eqref{G local law bounds} and use a simple Cauchy-Schwarz inequality to obtain 
  \begin{equation}\label{a b+n case}
    \abs{\partial_\gamma^l I_3^{(j)}} \lesssim \int_\C \abs{\Delta f_{z_j}^{(j)}(z)} \frac{1}{n}\int_{\eta_0}^T \sum_c \Bigl(\abs{G_{cb}}^2+\abs{G_{ca}}^2 \Bigr)\diff\eta \diff z. 
  \end{equation}
  Now it follows from the Ward-identity
  \begin{equation}\label{eq Ward}
    GG^\ast = G^\ast G = \frac{\Im G}{\eta}
  \end{equation}
  and the very crude bound \( \abs{G_{aa}}\lesssim 1 \) from~\eqref{G local law bounds} and \(\abs{\wh m}\lesssim 1\), that
  \[ \begin{split}
    \int_{\eta_0}^T \sum_c \Bigl(\abs{G_{cb}}^2+\abs{G_{ca}}^2 \Bigr)\diff\eta &= \int_{\eta_0}^T \frac{\abs{(\Im G)_{aa}}+\abs{(\Im G)_{bb}}}{\eta}\diff\eta \lesssim \int_{\eta_0}^T \frac{1}{\eta}\diff\eta\lesssim \log n. 
  \end{split}  \]
  By estimating the remaining \(z\)-integral in~\eqref{a b+n case} by \( n\norm[\big]{\Delta f^{(j)}}\) the claimed bound in~\eqref{fixed alpha bounds} for \(a=b+n\pmod{2n}\) follows. 
  
  In the case \(a\not\equiv b+n\pmod{2n}\) we can use~\eqref{G local law bounds} to gain a factor of \(\psi\) for some \(G_{ab}\) or \(G_{bb}-\wh m\) in all assignments except for the one in which all but two \(G\)-factors are diagonal, and those \(G_{aa},G_{bb}\)-factors are replaced by \(\wh m\). 
  For example, we would expand 
  \[
  G_{ca}G_{bb}G_{aa}G_{bc}=\wh m^2 G_{ca}G_{bc} + \wh m G_{ca}G_{bc} \landauO{\psi} + G_{ca}G_{bc} \landauO{\psi^2},
  \]
  where in all but the first term we gained at least a factor of \(\psi\). Using Cauchy-Schwarz as before we thus have the bound
  \begin{equation} \label{p2 I3}
    \begin{split}
      &\abs{\partial_\gamma^l I_3^{(j)}} \lesssim \int_\C \frac{\abs[\big]{\Delta f_{z_j}^{(j)}(z)}}{n} \biggl(\int_{\eta_0}^T \psi \sum_c \Bigl(\abs{G_{cb}}^2+\abs{G_{ca}}^2 \Bigr)\diff\eta \\&\qquad\qquad+\abs*{\int_{\eta_0}^T (\wh m)^{l-1} (G^2)_{aa}\diff\eta}+\abs*{\int_{\eta_0}^T (\wh m)^{l-1} (G^2)_{ab}\diff\eta} \biggr)\diff z,
    \end{split}
  \end{equation}
  where, strictly speaking, the second and third terms are only present for even, or respectively odd, \(l\). 
  For the first term in~\eqref{p2 I3} we again proceed by applying the Ward identity~\eqref{eq Ward}, and~\eqref{G local law bounds} to obtain the bound
  \[ \begin{split}
    \int_{\eta_0}^T \psi \sum_c \Bigl(\abs{G_{cb}}^2+\abs{G_{ca}}^2 \Bigr)\diff\eta &= \int_{\eta_0}^T \psi \frac{\abs{(\Im G)_{aa}}+\abs{(\Im G)_{bb}}}{\eta}\diff\eta \\
    & \lesssim \int_{\eta_0}^T \frac{\psi(\psi+\eta^{1/3})}{\eta}\diff\eta \lesssim \frac{\log n}{(n\eta_0)^2}.  
  \end{split}  \]
  For the second and third terms in~\eqref{p2 I3} we use \(\ii G^2=G'\), where prime denotes \(\partial_\eta\), and integration by parts, \(\abs{\wh m'}\lesssim \eta^{-2/3}\) from~\eqref{M Mp bounds}, and~\eqref{G local law bounds} to obtain the bounds
  \[\begin{split}
    \abs*{\int_{\eta_0}^T (\wh m)^{l-1} (G^2)_{aa}\diff\eta} &\lesssim \abs*{\int_{\eta_0}^T \wh m'(\wh m)^{l-2} G_{aa}\diff\eta} \\
    &\qquad +\abs{(\wh m(\ii\eta_0))^{l-1}  G(\ii\eta_0)_{aa}} + \abs{(\wh m(\ii T))^{l-1}  G(\ii T)_{aa}} \\
    &\lesssim  \abs*{\int_{\eta_0}^T \wh m' (\wh m)^{l-1}\diff \eta}  +\int_{\eta_0}^T \abs{\wh m'}\psi\diff\eta + \frac{1}{n^{1/4} (n\eta_0) }\\
    &\lesssim \frac{\log n}{n^{1/4} (n\eta_0) }
  \end{split} \]
  and 
  \[\begin{split}
    \abs*{\int_{\eta_0}^T (\wh m)^{l-1} (G^2)_{ab}\diff\eta} &\lesssim  \abs*{\int_{\eta_0}^T \wh m' (\wh m)^{l-2} G_{ab}\diff\eta} \\
    &\qquad + \abs{(\wh m(\ii\eta_0))^{l-1} G(\ii\eta_0)_{ab}}  + \abs{(\wh m(\ii T))^{l-1} G(\ii T)_{ab}} \\
    &\lesssim \int_{\eta_0}^T \abs{\wh m'}\psi\diff\eta  +  \frac{1}{n^{1/4} (n\eta_0) } \lesssim \frac{\log n}{n^{1/4} (n\eta_0) }.
  \end{split} \]
  In the explicit deterministic term we performed an integration and estimated 
  \[ \abs*{\int_{\eta_0}^T \wh m' (\wh m)^{l-1}\diff \eta} \lesssim \abs{\wh m(\ii\eta_0)}^l+\abs{\wh m(\ii T)}^l \lesssim n^{-l/4}+n^{-100}\le n^{-1/2}.\]
  The claim~\eqref{fixed alpha bounds} for \(l\ge 2\) and \(a\not\equiv b+n\pmod{2n}\) now follows from estimating the remaining \(z\)-integral in~\eqref{p2 I3} by \(n\norm[\big]{\Delta f^{(j)}}_1\). 
\end{proof}

\begin{proof}[Proof of Lemma~\ref{prop:impnewflow}]
  By~\eqref{eq:defbigflow} and Ito's Lemma it follows that 
  \begin{equation}
    \label{eq:itoimp}
    \E \frac{\diff Z_t}{\diff t}=\E \left[-\frac{1}{2}\sum_\alpha w_\alpha(t) \partial_\alpha Z_t+\frac{1}{2}\sum_{\alpha,\beta}\kappa_t(\alpha,\beta) \partial_\alpha\partial_\beta Z_t\right],
  \end{equation}
  where we recall the definition of \(\kappa_t\) in~\eqref{kappa t def}. In fact, the point-wise estimate from Lemma~\ref{lemma fixed alpha} gives a sufficiently strong bound for most terms in the cumulant expansion, the few remaining terms will be computed more carefully.
  
  In the cumulant expansion~\eqref{gen cum exp} of~\eqref{eq:itoimp} the second order terms cancel exactly and we now separately estimate the third-, fourth- and higher order terms. 
  \subsubsection*{Order three terms} 
  For the third order, when computing \(\partial_\alpha\partial_{\beta_1}\partial_{\beta_2} Z_t\) through the Leibniz rule we have to consider all possible assignments of derivatives \(\partial_\alpha,\partial_{\beta_1},\partial_{\beta_2}\) to the factors \smash{\(I_3^{(1)},\dots,I_3^{(k)}\)}. Since the particular functions \(f^{(j)}\) and complex parameters \(z_j\) play no role in the argument, there is no loss in generality in considering only the assignments
  \begin{equation}\label{3 assignments} 
    \begin{split}
      &\Bigl(\partial_{\alpha,\beta_1,\beta_2} I_3^{(1)}\Bigr)\prod_{j>1} I_3^{(j)},\;\; \Bigl(\partial_{\alpha,\beta_1} I_3^{(1)}\Bigr)\Bigl(\partial_{\beta_2} I_3^{(2)}\Bigr)\prod_{j>2} I_3^{(j)},\\ 
      &\Bigl(\partial_\alpha I_3^{(1)}\Bigr)\Bigl(\partial_{\beta_1} I_3^{(2)}\Bigr) \Bigl(\partial_{\beta_2} I_3^{(3)}\Bigr) \prod_{j>3} I_3^{(j)}
    \end{split}
  \end{equation}
  for the second and third term of which we obtain a bound of  
  \[ 
  \begin{split}
    & n^{\xi-3/2}e^{-3t/2} \biggl(\sum_{a\equiv b+n}  \prod_j \norm[\big]{\Delta f^{(j)}}_1 + \sum_{a\not\equiv b+n}  \prod_j \norm[\big]{\Delta f^{(j)}}_1 \frac{1}{(n\eta_0)^3} \biggr)\\
    &\qquad\lesssim \frac{n^\xi e^{-3t/2}}{n^{5/2}\eta_0^3}\prod_j \norm[\big]{\Delta f^{(j)}}_1
  \end{split} \]
  using Lemma~\ref{lemma fixed alpha} and the cumulant scaling~\eqref{eq:cumulant}. Note that the condition \(a\neq b\) in the lemma is ensured by the fact that for \(a=b\) the cumulants \(\kappa_t(\alpha,\beta_1,\dots)\) vanish. 
  
  The first term in~\eqref{3 assignments} requires an additional argument. We write out all possible index allocations and claim that ultimately we obtain the same bound, as for the other two terms in~\eqref{3 assignments}, i.e.
  \begin{equation}\label{J3}
    \begin{split}
      \abs{\sum_{\alpha\beta_1\beta_2} \kappa_t(\alpha,\beta_1,\beta_2) \partial_\alpha\partial_{\beta_1}\partial_{\beta_2} I_3^{(1)}} &\lesssim \frac{e^{-3t/2}}{n^{3/2}}\int_\C \frac{\abs[\big]{\Delta f_{z_1}^{(1)}}}{n} J_3 \diff z \\
      &\lesssim \frac{n^\xi e^{-3t/2}}{n^{5/2}\eta_0^3} \norm[\big]{\Delta f^{(1)}}_1
    \end{split}
  \end{equation}
  where 
  \begin{equation}\label{J3 crit}
    \begin{split}
      J_3 &:= \abs*{\int_{\eta_0}^T  \sum_{ab} (G^2)_{ab} G_{ab} G_{ab}\diff\eta }+\abs*{\int_{\eta_0}^T  \sum_{ab} (G^2)_{aa} G_{bb} G_{ab}\diff\eta }\\
      &\qquad+\abs*{\int_{\eta_0}^T  \sum_{ab} (G^2)_{ab} G_{aa} G_{bb}\diff\eta }.
    \end{split}
  \end{equation}   
  \begin{proof}[Proof of~\eqref{J3}]   
    Compared to the previous bound in Lemma~\ref{lemma fixed alpha} we now exploit the \(a,b\) summation via the isotropic structure of the bound in the local law~\eqref{local law ext}. We have the simple bounds
    \begin{equation}\label{eq iso G^2}
      \begin{split}
        \frac{\abs{\braket{\vx,\Im G\vx}}}{\norm{\vx}^2} &\lesssim \abs{\wh m}+n^\xi \psi \lesssim n\eta\psi^2,\\
        \abs{\braket{\vx, G^2 \vy}}&\le \frac{1}{\eta}\sqrt{\braket{\vx,\Im G\vx}\braket{\vy,\Im G\vy}}\lesssim n^\xi\norm{\vx}\norm{\vy} n\psi^2
      \end{split}
    \end{equation}
    as a consequence of the Ward identity~\eqref{eq Ward} and using~\eqref{eq local law} and~\eqref{m scaling}. 
    For the first term in~\eqref{J3 crit} we can thus use~\eqref{eq iso G^2} and~\eqref{eq Ward} to obtain
    \[\begin{split}  
      \abs*{\int_{\eta_0}^T  \sum_{ab} (G^2)_{ab} G_{ab} G_{ab}\diff\eta } & \lesssim n^\xi \int_{\eta_0}^T n \psi^2 \sum_{ab}\abs{G_{ab}}^2 \diff\eta \\
      & \lesssim n^\xi \int_{\eta_0}^T n \psi^2 \sum_a\frac{(\Im G)_{aa}}{\eta} \diff\eta\\
      & \lesssim n^\xi \int_{\eta_0}^T n^3 \psi^4\diff\eta \lesssim \frac{n^\xi}{n\eta_0^3}.
    \end{split}\] 
    For the second term in~\eqref{J3 crit} we split \(G_{bb}=\wh m+ \landauO{\psi}\) and bound it by 
    \[ \begin{split} 
      &\abs*{\int_{\eta_0}^T  \sum_{ab} (G^2)_{aa} G_{bb} G_{ab}\diff\eta } \\
      &\quad \lesssim n^\xi \int_{\eta_0}^T \psi \sum_{ab} \abs{ (G^2)_{aa} G_{ab}}\diff\eta  + \abs*{\int_{\eta_0}^T  \wh m \sum_{a} (G^2)_{aa} \braket{e_a,G\bm 1_{s(a)}}\diff\eta }\\ 
      &\quad\lesssim n^\xi \int_{\eta_0}^T n^{3/2} \psi^2 \biggl(\psi\sum_b \sqrt{\frac{(\Im G)_{bb}}{\eta}}+ \sqrt{\frac{ \braket{\bm 1_+,\Im G\bm1_+}+\braket{\bm 1_-,\Im G\bm1_-} }{\eta}}\biggr)\diff\eta  \\
      &\quad \lesssim n^\xi \int_{\eta_0}^T \Bigl(n^{3} \psi^4+n^{5/2}\psi^3 \Bigr)\diff\eta\lesssim \frac{n^\xi}{n\eta_0^3}
    \end{split}\]
    where \(e_a\) denotes the \(a\)-th standard basis vector, 
    \begin{equation}\label{bm1pm}
      \bm1_+ := (1,\ldots,1,0,\ldots,0), \quad \bm1_- := (0,\ldots,0,1,\ldots,1)
    \end{equation}
    are vectors of \(n\) ones and zeros, respectively, of norm \(\norm{\bm1_\pm}=\sqrt{n}\) and \(s(a):=-\) for \(a\le n\), and \(s(a):=+\) for \(a>n\). Here in the second step we used a Cauchy-Schwarz inequality for the \(a\)-summation in both integrals after estimating the \(G^2\)-terms using~\eqref{eq iso G^2}. Finally, for the third term in~\eqref{J3 crit} we split both \(G_{aa}=\wh m+\landauO{\psi}\) and \(G_{bb}=\wh m+\landauO{\psi}\) to estimate  
    \[ \begin{split}
      & \abs*{\int_{\eta_0}^T  \sum_{ab} (G^2)_{ab} G_{aa} G_{bb}\diff\eta } \\
      &\quad\lesssim n^\xi \int_{\eta_0}^T n^3 \psi^4\diff\eta + \sum_{a} \int_{\eta_0}^T  \abs{ \wh m  \braket{e_a, G^2 \bm 1_{s(a)}}\psi } \diff\eta  +\int_{\eta_0}^T \abs{\wh m^2 \braket{\bm 1_+, G^2\bm 1_-}}\diff\eta \\
      &\quad\lesssim \frac{n^\xi}{n\eta_0^3} +n^\xi \int_{\eta_0}^T  n^{5/2}\psi^3 \diff\eta +n^\xi \int_{\eta_0}^T \frac{n^2 \psi^2}{1+\eta^2}\diff\eta  \lesssim \frac{n^\xi}{n\eta_0^3},
    \end{split}\]
    using~\eqref{eq iso G^2}. In the last integral we used that \(\abs{\wh m}\lesssim (1+\eta)^{-1}\) to ensure the integrability in the large \(\eta\)-regime. Inserting these estimates on~\eqref{J3 crit} into~\eqref{J3} and estimating the remaining integral by \(n\norm[\big]{\Delta f^{(1)}}_1\) completes the proof of~\eqref{J3}.
  \end{proof}
  
  \subsubsection*{Order four terms}
  For the fourth-order Leibniz rule we have to consider the assignments
  \[ 
  \begin{split} 
    &\Bigl(\partial_{\alpha,\beta_1,\beta_2,\beta_3} I_3^{(1)}\Bigr)\prod_{j>1} I_3^{(j)},\;\; \Bigl(\partial_{\alpha,\beta_1,\beta_2} I_3^{(1)}\Bigr)\Bigl(\partial_{\beta_3} I_3^{(2)}\Bigr)\prod_{j>2} I_3^{(j)},\\ & \Bigl(\partial_{\alpha,\beta_1} I_3^{(1)}\Bigr)\Bigl(\partial_{\beta_2,\beta_3} I_3^{(2)}\Bigr) \prod_{j>2} I_3^{(j)}, \;\; \Bigl(\partial_{\alpha,\beta_1} I_3^{(1)}\Bigr)\Bigl(\partial_{\beta_2} I_3^{(2)}\Bigr)\Bigl(\partial_{\beta_3} I_3^{(3)}\Bigr) \prod_{j>3} I_3^{(j)},\\
    &\Bigl(\partial_{\alpha,\beta_1} I_3^{(1)}\Bigr)\Bigl(\partial_{\beta_2} I_3^{(2)}\Bigr)\Bigl(\partial_{\beta_2} I_3^{(3)}\Bigr)\Bigl(\partial_{\beta_3} I_3^{(4)}\Bigr) \prod_{j>4} I_3^{(j)},
  \end{split}
  \]
  for all of which we obtain a bound of 
  \[ \frac{n^{\xi} e^{-2t}}{n^2\eta_0^2} \prod_j \norm[\big]{\Delta f^{(j)}}_1,\]
  again using Lemma~\ref{lemma fixed alpha} and~\eqref{eq:cumulant}. 
  \subsubsection*{Higher order terms}
  For terms order at least \(5\), there is no need to additionally gain from any of the factors of \(I_3\) and we simply bound all those, and their derivatives, by \(n^\xi\) using Lemma~\ref{lemma fixed alpha}. This results in a bound of \(n^{\xi-(l-4)/2} e^{-l t/2}  \prod_j \norm[\big]{\Delta f^{(j)}}_1 \) for the terms of order \(l\). 
  
  By combining the estimates on the terms of order three, four and higher order derivatives, and integrating in \(t\) we obtain the bound~\eqref{eq:fingcr}. 
  This completes the proof of Lemma~\ref{prop:impnewflow}.
\end{proof} 

\appendix
\section{Extension of the local law}\label{app:aux}
\begin{proof}[Proof of Proposition~\ref{prop:local law}]
  The statement follows directly from~\cite[Theorem 5.2]{1907.13631} if \(\eta\ge\eta_0:= n^{-3/4+\epsilon}\). For smaller \(\eta_1\), using \(\partial_\eta G(\ii\eta)=\ii G^2(\ii\eta)\), we write 
  \begin{equation}\label{local law ext}
    \begin{split}
      \braket{\vx,[G(\ii\eta_1)-M(\ii\eta_1)]\vy} &= \braket{\vx, [G(\ii\eta_0)-M(\ii\eta_0)] \vy} \\
      &\quad+\ii \int_{\eta_0}^{\eta_1} \braket{\vx, [G^2(\ii\eta)-M'(\ii\eta)] \vy }\diff\eta
    \end{split} 
  \end{equation}
  and estimate the first term using the local law by \( n^{-1/4+\xi} \). For the second term we bound 
  \[
  \begin{split}
    \abs{\braket{\vx, G^2 \vy}} &\le \sqrt{ \braket{\vx,G^\ast G \vx}\braket{\vy,G^\ast G \vy} } = \frac{1}{\eta} \sqrt{\braket{\vx,\Im G \vx}\braket{\vy,\Im G \vy}}, \\ 
    \abs{\braket{\vx,M' \vy}}&\lesssim \norm{\vx}\norm{\vy}\frac{1}{\eta^{2/3}}
  \end{split}\]
  from \(\norm{M'}\lesssim (\Im\wh m)^{-2}\) and~\eqref{m scaling},
  and use monotonicity of \(\eta\mapsto \eta \braket{\vx,\Im G(\ii\eta)\vx}\) in the form 
  \[\Im \braket{\vx, G(\ii\eta)\vx} \le \frac{\eta_0}{\eta} \braket{\vx, \Im G(\ii\eta_0)\vx}  \prec \norm{\vx}^2 \Bigl(\frac{\eta_0^{4/3}}{\eta} + \frac{\eta_0^{2/3}}{\eta n^{1/2}} \Bigr) \lesssim\norm{\vx}^2 \frac{n^{4\epsilon/3}}{n\eta}.\]
  After integration we thus obtain a bound of \( \norm{\vx}\norm{\vy}n^{4\epsilon/3}/(n\eta_1)\) which proves the first bound in~\eqref{eq local law}. The second, averaged, bound in~\eqref{eq local law} follows directly from the first one since below the scale \(\eta\le n^{-3/4}\) there is no additional gain from the averaging, as compared to the isotropic bound. 
  
  In order to conclude the local law simultaneously in all \(z,\eta\) we use a standard \emph{grid argument}. To do so, we choose a regular grid of \(z\)'s and \(\eta\)'s at a distance of, say, \(n^{-3}\) and use Lipschitz continuity (with Lipschitz constant \(n^2\)) of \((\eta,z)\mapsto G^z(\ii\eta)\) and a union bound over the exceptional events at each grid point. 
\end{proof}

\printbibliography{}
\end{document}